\documentclass[11pt]{amsart}

\usepackage[utf8]{inputenc}
\usepackage[T1]{fontenc}
\usepackage[english]{babel}

\usepackage{amssymb}
\usepackage{mathrsfs}
\usepackage{amsfonts}
\usepackage{amssymb,mathtools}

\usepackage{amsmath,amsthm,bbm}

\usepackage{graphicx}

\usepackage[dvipsnames]{xcolor}

\numberwithin{equation}{section}
\usepackage[nodayofweek]{datetime}

			\newcommand{\dd}{\mathrm{d}}

			\newcommand{\du}{\mathrm{d}u}
			\newcommand{\dx}{\mathrm{d}x}

			\newcommand{\1}{\mathbbm{1}}

			\def\pf{{\bf Proof: }}

			\def\om{\omega}
			\def\Om{\Omega}

			\def\un{\infty}
			\def\eh{\frac{1}{2}}

			\def\zp{2 \pi}
			\def\epf{$\Box$ \\}

			\def\bx{\overline{x}}
			
			\def\bxi{\overline{\xi}}
			\def\beeta{\overline{\eta}}

			\newtheorem{thm}{Theorem}[section]
			\newtheorem{pr}[thm]{Proposition}
			\newtheorem{co}[thm]{Corollary}
			\newtheorem{lem}[thm]{Lemma}
			\theoremstyle{definition}

			\newcommand{\N}{\mathbb{N}}
			\newcommand{\R}{\mathbb{R}}
			\newcommand{\Z}{\mathbb{Z}}
			
			\newcommand{\be}{\begin{equation}}
			\newcommand{\ee}{\end{equation}}
			\newcommand{\bdm}{\begin{displaymath}}
			\newcommand{\edm}{\end{displaymath}}
			\newcommand{\bean}{\begin{eqnarray}}
			\newcommand{\eean}{\end{eqnarray}}
			\newcommand{\bea}{\begin{eqnarray*}}
				\newcommand{\eea}{\end{eqnarray*}}

			\newcommand{\cB}{\mathcal{B}}

\author[P. Imkeller]{Peter Imkeller}
\address[P. Imkeller]{Humboldt-Universit\"at zu Berlin, Institut f\"ur Mathematik, Unter den Linden 6, D-10099 Berlin, Germany}
\email{\tt imkeller@mathematik.hu-berlin.de}
\thanks{P.~Imkeller was supported in part by DFG Research Unit FOR 2402.}

\author[O. Menoukeu Pamen]{Olivier Menoukeu Pamen}
\address[O. Menoukeu Pamen]{African Institute for Mathematical Sciences, Ghana, and Institute for Financial and Actuarial Mathematics, Department of Mathematical Sciences, University of Liverpool, L69 7ZL, United Kingdom}
\email{\tt menoukeu@liv.ac.uk}
\thanks{O.~Menoukeu Pamen acknowledges the funding provided by the Alexander von Humboldt Foundation, under the programme financed by the German Federal Ministry of Education and Research entitled German Research Chair No 01DG15010.}

\author[G. dos Reis]{Gon\c{c}alo dos Reis}
\address[G. dos Reis]{School of Mathematics, University of Edinburgh, Peter Guthrie Tait Road, Edinburgh, EH9 3FD, United Kingdom, and Centro de Matem\'atica e Aplica\c c$\tilde{\text{o}}$es (CMA), FCT, UNL, Portugal
}
\email{\tt G.dosReis@ed.ac.uk}
\thanks{G.~dos Reis acknowledges support from the \emph{Funda{\c c}$\tilde{\text{a}}$o para a Ci$\hat{e}$ncia e a Tecnologia} (Portuguese Foundation for Science and Technology) through the project UIDB/00297/2020 (Centro de Matem\'atica e Aplica\c c$\tilde{\text{o}}$es CMA/FCT/UNL)}

\author[A. R\'eveillac]{Anthony R\'eveillac}
\address[A. R\'eveillac]{INSA, D\'epartement de G\'enie Math\'ematique, 135 avenue de Rangueil, 31077 Toulouse Cedex 4, France}
\email{\tt anthony.reveillac@insa-toulouse.fr}

\thanks{$^*$\textbf{Dedicated to Professor Peter Kloeden  on the occasion of  his 70$^{\textrm{th}}$ birthday}}

\title[Rough Weierstrass functions and smoothness of their SBR measure ]{Rough Weierstrass functions and dynamical systems: the smoothness of the SBR measure}

\date{
       \currenttime,
       \today
}


				
\begin{document}
				\selectlanguage{english}

				\begin{abstract}
We investigate Weierstrass functions with roughness parameter $\gamma$ that are H\"older continuous with coefficient $H={\log\gamma}/{\log \eh}.$ Analytical access is provided by an embedding into a dynamical system related to the baker transform where the graphs of the functions are identified as their global attractors. They possess stable manifolds hosting Sinai-Bowen-Ruelle (SBR) measures. We systematically exploit a telescoping property of associated measures to give an alternative proof of the absolute continuity of the SBR measure for large enough $\gamma$ with square integrable density. Telescoping allows a macroscopic argument using the transversality of the flow related to the mapping describing the stable manifold. The smoothness of the SBR measure can be used to compute the Hausdorff dimension of the graphs of the original Weierstrass functions and investigate their local times.
					\end{abstract}
				
				\maketitle

		\section{Introduction}
				
The interest in the subject of this paper, rough Weierstrass curves, arose from a two dimensional example of such functions studied in the context of the Fourier analytic approach of rough path analysis or rough integration theory laid out in \cite{gubinelliimkellerperkowski2015} and \cite{gubinelliimkellerperkowski2016}. In \cite{gubinelliimkellerperkowski2016}, the construction of a Stratonovich type integral of a rough function $f$ with respect to another rough function $g$ is based on the notion of paracontrol of $f$ by $g$. This Fourier analytic concept generalizes the original notion of control introduced by Gubinelli \cite{gubinelli2004}. In search of a good example of two-dimensional functions for which no component is controlled by the other one, in \cite{imkellerproemel15} we come up with a pair of Weierstrass functions $W=(W_1, W_2)$. One of them fluctuates on all dyadic scales in a sinusoidal manner, the other one in a cosinusoidal one. Hence while the first one has minimal increments, the second one has maximal ones, and vice versa. This is seen to mathematically underpin in a rigorous way the fact that they are mutually not controlled. It is also seen that the L\'evy areas of the approximating finite sums of the representing series do not converge. This geometric pathology motivated us to look for further geometric properties of the pair, or of its single components. In a companion paper \cite{IdR2018} we investigate the question: if L\'evy's area fails to exist, is $W$ space filling, at least at a nontrivial portion of its graph? And what is its Hausdorff dimension? Here we concentrate on the cosinusoidal one-dimensional component, given by
$$W(x) = \sum_{n=0}^\infty \gamma^n \cos(\zp 2^n x),\quad x\in[0,1],$$
with a roughness parameter $\gamma\in]\eh,1[$ (see Figure \ref{fig:idrRangeW-v01}). It is H\"older continuous with Hurst parameter $H = {\log \gamma}/{\log \eh}.$

\begin{figure}[htbp]
\centering
				\includegraphics[width=0.6\textwidth]{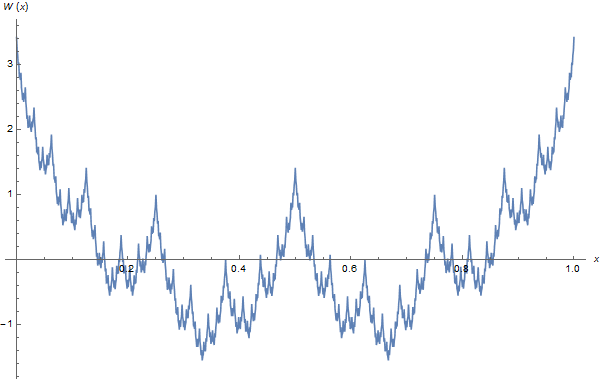}
					\caption{Graph of $W$ for $x\in[0,1]$ and $\gamma=1/\sqrt{2}$; $\{(x,W(x)):x\in[0,1]\}\subset \R^2$.}
					\label{fig:idrRangeW-v01}
\end{figure}


It had been noticed in a series of papers (see \cite{hunt1998}, \cite{baranski02}, \cite{baranski12}, \cite{baranski14published}, \cite{baranski15survey}, \cite{keller17-Publication-of-2015}, \cite{shen2018}) on one-dimensional Weierstrass type curves that the number of iterations of the expansion by a real factor can be taken as a starting point in interpreting their graphs as pullback attractors of dynamical systems in which a baker transformation defines the dynamics. This observation marks, in many of the papers quoted, the point of departure for determining the Hausdorff dimension of graphs of one dimensional Weierstrass type functions. For a historical survey of this work the reader may consult \cite{baranski14published}. For our curve we use the same metric dynamical system based on a suitable baker transformation as a starting point. This is done by introducing, besides a variable $x$ that encodes expansion by the factor $2$ forward in time, an auxiliary variable $\xi$ describing contraction by the factor $\frac{1}{2}$ in turn, forward in time as well. The operation of expansion-contraction in both variables is described by the baker transformation $B=(B_1,B_2)$ (see Lemma \ref{l:action_B} for details). Backward in time, the sense of expansion and contraction is interchanged. The action of applying forward expansion in one step just corresponds to stepping from one term in the series expansion of $W$ to the following one. This indicates that $W$ is an attractor of a three dimensional hyperbolic dynamical system $F$ that, besides contracting a leading variable by the factor $\gamma$, adds the first term of the series to the result. For details, see Lemma \ref{l:attractor}. So by definition of $F$, $W$ is its attractor. Since $\frac{1}{2}$, the factor $x$ in the forward fiber motion, is the smallest Lyapunov exponent of the linearization of $F$, there is a stable manifold related to this Lyapunov exponent. It is spanned by the vector which is given as another Weierstrass type series
$$S(\xi,x)= \zp\sum_{n=1}^\infty \kappa^n \sin\big(\zp B^n_2(\xi,x)\big),$$
where $\kappa = \frac{1}{2\gamma}\in]\eh,1[$ is a roughness parameter dual to $\gamma.$ This will be explained below. The pushforward of the Lebesgue measure by $S(\cdot, x)$ for $x\in[0,1]$ fixed, is the $x$-marginal of the Sinai-Bowen-Ruelle measure of $F$. The definition of $F$ as a linear transformation added to a very smooth function may be understood as conveying the concept of \emph{self-affinity} for the Weierstrass curve. Self-affinity can be seen as a concept providing the magnifying lens to zoom out microscopic properties of the underlying geometric object to a macroscopic scale. Our main tool of \emph{telescoping relations} translates this rough idea into mathematical formulas, quite in the sense of Keller's paper \cite{keller17-Publication-of-2015}.
The main goal of this paper is the study of smoothness of the Sinai-Bowen-Ruelle measure by means of telescoping. Its absolute continuity was seen in many papers to allow information on geometric properties of the underlying Weierstrass curves such as their Hausdorff dimensions (see for instance Keller \cite{keller17-Publication-of-2015}, or Baranski \cite{baranski15survey}). To investigate this smoothness, we will start by looking at the measure $\rho$ given by the pushforward of three-dimensional Lebesgue measure by the transformation
$$(\xi,\eta,x)\mapsto S(\xi,x)-S(\eta,x).$$
We shall derive \emph{telescoping equations} relating $\rho$ with a macroscopic version $\hat{\rho}$ living on the macroscopic set $\{\eh<|\xi-\eta|\}$ contained in $\{\bxi_0\not=\beeta_0\}$, where $\bxi_{0}$ resp.~$\beeta_0$ denote the first components in the dyadic expansion of $\xi$ resp.~$\eta.$ The key element of our approach is the observation that the behaviour of $S$ on these macroscopic sets is easy to describe, at least for specific ranges of the roughness parameters $\kappa$ resp.~$\gamma$. The related macroscopic properties are close to \emph{transversality} properties used in many of the papers cited at the beginning of this paragraph. The first appearance of this notion describing a quality of the flow related to the map $x\mapsto S(\xi,x)-S(\eta,x)$ for $(\xi,\eta)$ in the macroscopic set $\{\bxi_0\not= \beeta_0\}$ is in Tsujii \cite{tsujii01}. We shall basically refine the transversality notion to properties of the maps $x\mapsto S(\xi,x)-S(\eta,x)$ on the macroscopic set $\{\eh<|\xi-\eta|\}$. For instance, we shall show that the maps possess strict local minima below the axis, at most two strictly simple roots, for ranges of $\kappa$ given by intervals to the right of $\eh$, resp.~ranges of $\gamma$ to the left of $1$. These ranges are not optimal in our construction, since we obtain them from an approximation of the series representing $S$ by its first three terms, and a global estimate for the remainder. Better estimates for the ranges are expected for more exact approximations of $S$, but become much more involved with each additional term of the expansion. The properties have roughly one common denominator, namely that roots of the function do not coincide with roots of its derivative, a property that immediately leads to classical transversality. They are simple to formulate, but tedious to prove.

For this purpose, as opposed to other papers, we use a new formula of representation of $S(\xi,\cdot)-S(\eta,\cdot)$ and its derivatives, based on the increment function
$$g(x) = 4\pi \sum_{m=0}^\infty \kappa^m \sin \Big(\frac{\pi}{2^{m+1}}\Big) \cos\Big( \pi \frac{1+2x}{2^{m+1}}\Big),\quad x\in[0,1].$$ For $\eta=0$ it reads
$$S(\xi,x)-S(0,x)= \sum^{\infty}_{\ell=1}\kappa^{\tau_\ell+1}g\big(B_2^{\tau_\ell}(\xi,x)\big),$$
where $\tau_l$ are the times at which the dyadic sequence related to $\xi$ has an upward jump from 0 to 1. The sequence $(\tau_{\ell})_{\ell\ge 1}$ completely characterizes $\xi.$ The formula allows to study the geometric properties of $S(\xi,\cdot)-S(\eta,\cdot)$ (and its derivatives) by the geometric properties of $g$ (and its derivatives). Roughly, $S(\xi,\cdot)-S(\eta,\cdot)$ is a relatively small perturbation of $g$, varying with $\xi$, as will be illustrated by a number of graphs obtained via \emph{Mathematica} simulations. Therefore the treatment of the properties of the function $g$ is crucial for our analysis of transversality. $g$ will be seen to be almost convex, i.e.~convex above $x_0\le .027,$ with exactly one root, and one strict local minimum. How this property is inherited by $S(\xi,\cdot)-S(\eta,\cdot)$ is subject of a detailed and still complex discussion below, that leads to the conclusion that it is convex on a large interval with a complement consisting of small intervals near $0$ and $1$, and that it has a unique global minimum. Showing that this minimum is below the $x$-axis will constitute the main step in the transversality proof, and will determine the interval of $\kappa$ for which our smoothness results on the SBR measure are valid. We conjecture that the representation formula by means of upward jump times of dyadic expansions has the potential to shed more light on the complex arguments used in \cite{shen2018} extending the concept of transversality that ultimately leads to a proof of absolute continuity of the SBR measure for the entire parameter interval.

Via the telescoping identity, we will be able to explicitly describe densities of $\rho$. They relate to the functions mentioned on the respective macroscopic sets, where transversality leads to smoothness and boundedness properties of the densities. These are finally used in a Fourier analytic criterion for the smoothness of the SBR measure to deduce its absolute continuity in the main theorem (Theorem \ref{t:SBR_ac}) of the paper. This result just partially recovers a result proved in Theorem 1.2 of \cite{shen2018} by means of a different extension of Tsujii's \cite{tsujii01} tools of transversality, and complex methods different from our telescoping concept. Existence and square integrability of densities of the SBR measure are crucial for estimates on the Hausdorff dimension of Weierstrass curves in \cite{keller17-Publication-of-2015} or \cite{baranski15survey}.
We conjecture that our methods for proving transversality have the potential to lead to further fine structure geometric properties of Weierstrass curves such as their local times.

The paper is organized along these lines of reasoning in the following way. In Section \ref{s:attractor}, repeating \cite{baranski02}, \cite{hunt1998} or \cite{keller17-Publication-of-2015}, we explain the interpretation of our Weierstrass curve in terms of dynamical systems based on the baker transform. In Section \ref{s:SBR}, we describe the measures related to the SBR measure, and establish the representation of $S$ by the expansion along the upward jumps in the dyadic sequences characterizing $\xi$. In section \ref{s:transversality} we establish the transversality of $S(\xi,\cdot)-S(\eta,\cdot)$, based on a thorough analysis of the geometric properties of the increment function $g$. Section \ref{s:rho_rhohat} is devoted to establishing the telescoping relationship between $\rho$ and its macroscopic restriction. Finally, in section \ref{s:abs_continuity} we prove our main theorem on the absolute continuity of the SBR measure.




				\section{The curve as the attractor of a dynamical system}\label{s:attractor}
				
				Let $\gamma\in]\eh,1[.$ Our aim is to investigate the fine structure geometry of the one-dimensional Weierstrass curves given by
				\be
				\label{eq:DefinitionW}
				W(x) = \sum_{n=0}^\infty \gamma^n \cos (\zp 2^n x),\quad x\in[0,1].
				\ee
				
				Let us first determine the H\"older exponent of $x\mapsto W(x)$ (see \cite{baranski15survey} for an overview).
				\begin{pr}\label{p:hoelderexponent}
					$W$ is H\"older continuous with exponent $-\frac{\log \gamma}{\log 2}$.
				\end{pr}
				
				\begin{proof}
					Let $x,y\in[0,1]$ and choose an integer $k\ge 0$ such that
					$$2^{-(k+1)}\leq |x-y|\leq 2^{-k}.$$
					Then we have, using the Lipschitz continuity of the $\cos$ function
					\bea
					|W(x) - W(y)| &\leq& \sum_{n=1}^k \gamma^n |\cos(\zp 2^n x) - \cos(\zp 2^n y)|
					+ 2\, \sum_{n=k+1}^\infty \gamma^n \\
					&\lesssim& \sum_{n=1}^k (2\gamma)^n |x-y| + \gamma^k \lesssim (2\gamma)^k\,\,2^{-k} + \gamma^k \simeq \gamma^k = 2^{-k \frac{\log \gamma}{\log \eh}}\\
					&\lesssim& |x-y|^{-\frac{\log \gamma}{\log 2}}.
					\eea
					This shows that $\frac{\log \gamma}{\log \eh}$ is an upper bound for the H\"older exponent of $W$. To see that it is also a lower bound, for $n\in\N$ choose $x_n = 0, y_n = 2^{-n}.$ Then we may write
					\bea
					|W(x_n) - W(y_n)| &=& \Big|\sum_{k=1}^\infty \gamma^k \big(1-\cos(\zp 2^{k-n})\big)\Big|
					\\
					&=& \sum_{k=1}^{n-1} \gamma^k \big(1-\cos(\zp 2^{k-n})\big) \gtrsim \gamma^{n} = 2^{-n \frac{\log \gamma}{\log \eh}}
					=|x_n-y_n|^{-\frac{\log \gamma}{\log 2}}.
					\eea
					Since $|x_n-y_n|\to 0$ as $n\to\infty$, this shows that $-\frac{\log \gamma}{\log 2}$ is also a lower bound for the H\"older exponent of $W$. The argument can be extended to the other points in the interval.
				\end{proof}
				

				Our access to the analysis and geometry of $W$ is via the theory of dynamical systems. In fact, we shall describe a dynamical system on $[0,1]^2$, alternatively $\Om = \{0,1\}^\N\times\{0,1\}^\N$, which produces the graph of the function as its attractor.
				For elements of $\Om$ we write for convenience $\om = ((\om_{-n})_{n\ge 0}, (\om_n)_{n\ge 1})$; one understands $\Omega$ as the space of $2$-dimensional sequences of Bernoulli random variables.  Denote by $\theta$ the canonical shift on $\Om$, given by
				$$\theta:\Om\to\Om,\quad \om\mapsto (\om_{n+1})_{n\in\Z}.$$
				$\Om$ is endowed with the product $\sigma$-algebra, and the infinite product $\iota = \otimes_{n\in\Z} (\eh \delta_{\{0\}} + \eh \delta_{\{1\}})$ of Bernoulli measures on $\{0,1\}.$ We recall that $\iota$ is $\theta$-invariant.
				
				Now let
				$$T=(T_1,T_2):\Om \to [0,1]^2,\quad \om \mapsto (\sum_{n=0}^\infty \om_{-n} 2^{-(n+1)}, \sum_{n=1}^\infty \om_n 2^{-n}).$$ Denote by $T_1$ the first component of $T$, and by $T_2$ the second one.
				It is well known that $\iota$ is mapped under the transformation $T$ to $\lambda^2$ (i.e.~$\iota=\lambda^2\circ T$), the 2-dimensional Lebesgue measure. It is also well known that the inverse of $T$, the dyadic representation of the two components from $[0,1]^2$, is uniquely defined apart from the dyadic pairs. For these we define the inverse to map to the sequences not converging to $0$. Let
				$$B =(B_1,B_2) = T\circ\theta \circ T^{-1}.$$
				We call $B=(B_1,B_2)$ the \emph{baker's transformation}. The $\theta$-invariance of $\iota$ directly translates into the $B$-invariance of $\lambda^2$:
				\begin{align}
				\label{eq:InvarianceBakerB}
					\lambda^2 \circ B^{-1} = (\lambda^2\circ T) \circ \theta^{-1}\circ T^{-1} = (\iota\circ \theta^{-1}) \circ T^{-1} = \iota \circ T^{-1} = \lambda^2.
				\end{align}
				For $(\xi, x)\in[0,1]^2$ let us denote
				$$T^{-1}(\xi, x) = \big((\bxi_{-n})_{n\ge 0}, (\bx_n)_{n\ge 1}\big).$$
				Let us calculate the action of $B$ and its integer iterates on $[0,1]^2.$
				
				\begin{lem}\label{l:action_B}
					Let $(\xi, x) \in[0,1]^2$. Then for $k\ge 0$
					$$
					B^k(\xi, x)
					=
					\Big(2^k \xi\, (\textrm{mod } 1), \frac{\bxi_{-k+1}}{2} + \frac{\bxi_{-k+2}}{2^2}+\cdots+\frac{\bxi_0}{2^k} + \frac{x}{2^k}\Big),
					$$
					for $k\ge 1$
					$$ B^{-k}(\xi, x)
					=
					\Big(\frac{\xi}{2^k}+ \frac{\bx_1}{2^k} + \frac{\bx_2}{2^{k-1}}+\cdots+ \frac{\bx_k}{2}, 2^k x (\textrm{mod } 1)\Big).
					$$
				\end{lem}
				\pf
				By definition of $\theta^k$ for $k\ge 0$
				$$B^k(\xi,x) = \Big(\sum_{n\ge 0} \bxi_{-n+k} 2^{-(n+1)}, \frac{\bxi_{-k+1}}{2} + \frac{\bxi_{-k+2}}{2^2}+\cdots+\frac{\bxi_0}{2^k} + \sum_{n\ge 1} \bx_n 2^{-(k+n)}\Big).$$
				Now we can write
				$$\sum_{n\ge 0} \bxi_{-n+k} 2^{-(n+1)} = 2^k \xi (\mbox{mod } 1)
				\quad\textrm{and}\quad
				\sum_{n\ge 1} \bx_n 2^{-(k+n)} = \frac{x}{2^k}.$$
				This gives the first formula.
				For the second, note that by definition of $\theta^{-k}$ for $k\ge 1$
				$$B^{-k}(\xi, x) = \Big(\sum_{n\ge 0} \bxi_{-n} 2^{-(n+1+k)} + \frac{\bx_1}{2^k} + \frac{\bx_2}{2^{k-1}}+\cdots +\frac{\bx_k}{2}, \sum_{n\ge 1} \bx_{n+k} 2^{-n}\Big).$$
				Again, we identify
				$$\sum_{n\ge 1} \bx_{n+k} 2^{-n} = 2^k x (\textrm{mod } 1)
				\qquad\textrm{and}\qquad
				\sum_{n\ge 0} \bxi_{-n} 2^{-(n+1+k)} = \frac{\xi}{2^k}.$$
				\epf
				For $k\in\Z, (\xi, x)\in[0,1]^2$ we abbreviate the $k$-fold iterate of the baker transform of $(\xi,x)$ as
				$$B^{k}(\xi, x) = \big( B^{k}_1(\xi, x), B^{k}_2(\xi, x) \big) = (\xi_k, x_k),$$
				where for $k\ge 0$
				$$\xi_k = 2^k\xi (\textrm{mod } 1),
				\quad\textrm{and}\quad
				x_k = \frac{\bxi_{-k+1}}{2} + \frac{\bxi_{-k+2}}{2^2}+\cdots+\frac{\bxi_0}{2^k} + \frac{x}{2^k},$$
				and for $k\ge 1$
				$$\xi_{-k} = \frac{\xi}{2^k}+ \frac{\bx_1}{2^k} + \frac{\bx_2}{2^{k-1}}+\cdots + \frac{\bx_k}{2},
				\quad\textrm{and}\quad x_{-k} = 2^k x (\mbox{mod } 1).$$
				Following Baranski \cite{baranski02, baranski12,baranski14published}, Shen \cite{shen2018}, Hunt \cite{hunt1998}, \cite{ledrappier92}, and \cite{IdR2018}, we will next interpret the Weierstrass curve $W$ by a transformation on our base space $[0,1]^2$.
				Let
				\bea
				F: [0,1]^2\times\R &\to& [0,1]^2\times \R,\\
				(\xi,x, y)&\mapsto& \Big(B(\xi, x), \gamma y + \cos \big(\zp B_2(\xi, x)\big)\Big).
				\eea
				Here we note $B = (B_1, B_2)$ for the two components of the baker transform $B$.
				
				For convenience, we extend $W$ from $[0,1]$ to $[0,1]^2$ by setting
				$$W(\xi, x) = W(x),\quad \xi, x\in[0,1].$$
				To see that the graph of $W$ is an attractor for $F$, the skew-product structure of $F$ with respect to $B$ plays a crucial role.
				\begin{lem}\label{l:attractor}
					For any $\xi, x\in[0,1]$ we have
					$$F\big(\xi, x, W(\xi, x)\big) = \Big(B(\xi, x), W\big(B(\xi, x)\big)\Big).$$
				\end{lem}
				
				\pf By the definition of the baker's transform we may write
				$$W(\xi, x) = \sum_{n=0}^\infty \gamma^n \cos \Big(\zp B^{-n}_2(\xi,x)\Big),\quad \xi, x\in[0,1].$$
				Hence, setting $k=n-1$, for $\xi, x\in[0,1]$
				\bea
				W\big(B_2(\xi,x)\big)
				&=& \sum_{n=0}^\un \gamma^n \cos \Big(\zp B^{-n+1}_2(\xi,x)\Big)\\
				&=& \cos \big(\zp B_2(\xi, x)\big) + \gamma \sum_{k=0}^\infty \gamma^k \cos \Big(\zp B^{-k}_2(\xi,x)\Big)\\
				&=& \cos \big(\zp B_2(\xi, x)\big) + \gamma W(x).
				\eea
				Hence by definition of $F$
				$$\Big(B(\xi, x), W(B(\xi, x))\Big)
				= \Big(B(\xi, x), W(B_2(\xi, x))\Big) = F\Big(\xi, x, W(\xi, x)\Big).$$
				\epf
				
				To assess the stability properties of the dynamical system generated by $F$, let us calculate its Jacobian. We obtain for $\xi, x \in[0,1], y \in\R$
				\bea
				DF(\xi, x, y) &=&
				\left[
				\begin{array}{ccc}
					2&0&0\\
					0&\eh&0\\
					0& -\pi\sin \big(\zp  B_2(\xi, x) \big)& \gamma
				\end{array}\right].
				\eea
				Hence the Lyapunov exponents of the dynamical system associated with $F$ are given by $2, \eh$, and $\gamma.$
				The corresponding invariant vector fields are given by
				$$\left(\begin{array}{c} 1\\0\\0\end{array}\right),\quad X(\xi,x)
				= \left(\begin{array}{c}0\\1\\  \zp \sum_{n=1}^\un \big(\frac{1}{2\gamma}\big)^n \sin \big(\zp B^n_2(\xi,x)\big)\end{array}\right),\quad \left(\begin{array}{c} 0\\0\\1\end{array}\right),$$
				as is straightforwardly verified. Note that $X$ is well defined, since by our choice of $\gamma$ we have $2\gamma > 1$.
				Hence we have in particular for $\xi, x\in[0,1], y \in\R$
				$$DF(\xi, x, y) X(\xi, x) = \frac{1}{2}\,\, X\big(B(\xi, x)\big).$$
				Note that the vector $X$ spans an invariant stable manifold and does not depend on $y$.

				\section{The Sinai-Bowen-Ruelle measure}\label{s:SBR}
				
				Abbreviate $\kappa = \frac{1}{2\gamma}\in ]0,1[.$ In Tsujii \cite{tsujii01} the problem of the absolute continuity of the Sinai-Bowen-Ruelle (SBR) measure on the stable manifold described by
				$$S(\xi, x) = \zp \sum_{n=1}^\un \kappa^n \sin \big( \zp B^n_2(\xi,x) \big),\quad \xi, x\in[0,1],$$
				with respect to Lebesgue measure has been treated. It has been related to the \emph{transversality} of the map $x\mapsto S(\xi,x)-S(\eta,x)$ for $\xi, \eta\in[0,1]$ such that $\bxi_0 \not= \beeta_0$. We shall now tackle a proof of this statement for a reasonably big range of $\kappa$ by giving the problem of transversality of $S$ a closer look. Our proof rests upon a comparison of the measures $\rho$, image measure of three dimensional Lebesgue measure under the map
				$$(x,\xi,\eta)\mapsto S(\xi,x)-S(\eta,x),$$
				and its restriction to the set $\{\eh<|\xi-\eta|\}$ which contains $\{\bxi_0\not= \beeta_0\}$, namely $\hat{\rho} = \rho(\cdot\cap\{\eh<|\xi-\eta|\})$.
%
%
				This comparison will simplify the derivation of smoothness of the SBR measure from transversality in the spirit of Tsujii \cite{tsujii01}.
			
				To recall the SBR measure of $F$, let us first calculate the action of $S$ on the $\lambda^2$-measure preserving map $B$. For $\xi, x\in[0,1]$  we have
				\bea
				S(B(\xi, x)) &=& \zp \sum_{n=1}^\un \kappa^n \sin \Big(\zp B^n_2\big(B_2(\xi, x)\big)\Big)\\
				&=& \zp \sum_{n=1}^\un \kappa^n \sin \big(\zp B_2^{n+1}(\xi, x)\big)\\
				&=& \zp \kappa^{-1} \sum_{k=1}^\un \kappa^k \sin \big(\zp B^k_2(\xi, x)\big) - \zp \sin \big(\zp B_2(\xi, x)\big)\\
				&=& 2 \gamma S(\xi, x) - \zp \sin\big(\zp B_2(\xi, x)\big).
				\eea
				So we may define the Anosov skew product
				\begin{align*}
				\Gamma:[0,1]^2\times \R
				& \to[0,1]^2\times \R,
				\\
				(\xi, x, v)
				&\mapsto
				\Big(B(\xi, x), 2 \gamma v - \zp \sin\big( \zp B_2(\xi, x)\big)\Big).
				\end{align*}
				Then the equation just obtained yields the following result ({compare with Lemma \ref{l:attractor}}).
				\begin{lem}\label{l:SBR}
					For $\xi, x\in[0,1]$ we have
					$$\Gamma\big(\xi, x, S(\xi, x)\big) = \big(B(\xi, x), S(B(\xi, x))\big).$$
					The push-forward measure of the Lebesgue measure in $\R^2$ to the graph of $S$ given by
					$$\psi = \lambda^2\circ (\mbox{id}, S)^{-1}$$
					on $\cB([0,1]^2)\otimes \cB(\R)$ is $\Gamma$-invariant.
				\end{lem}
				
				\pf The first equation has been verified above. The $\Gamma$-invariance of $\psi$ is a direct consequence of the $B$-invariance of $\lambda^2.$ \epf
				
				Define $\pi_2:[0,1]^2\to[0,1], (\xi, x)\mapsto x$ and define the measure
				\begin{align}
				\label{eq:SBRforGamma}
				\mu & = \lambda^2\circ (\pi_2, S)^{-1}
				\end{align}
				on $\cB([0,1]^2)$. The measure $\mu$ is called the \emph{Sinai-Bowen-Ruelle measure} of $\Gamma$. Its marginals in $x\in[0,1]$ are denoted $\mu_x=\lambda \circ S(\cdot,x)^{-1}$.

				
				We now define a map on our probability space that exhibits certain increments of $S$ in a self similar way.
				Let
				
				$$G(\xi,x) = \zp\, \sum_{n\in\Z} \kappa^{-n} \big[\sin\big( \zp B^{-n}_2(\xi, x)\big) - \sin \big(\zp B^{-n}_2(0, x)\big)\big],\quad \xi, x\in[0,1].$$
				
				Then we have the following simple relationship between $G$ and $S$.
				
				\begin{lem}\label{l:doubly_infinite_series_G}
					For $x, \xi, \eta\in [0,1]$ we have
					$$G(\xi, x) - G(\eta, x) = S(\xi,x) - S(\eta, x).$$
				\end{lem}
				
				\pf
				For $x, \xi, \eta\in[0,1]$ we have indeed
				\bea
				G(\xi, x) - G(\eta, x) &=& \sum_{n\in\Z} \kappa^{-n} \big[\sin \big(\zp B^{-n}_2(\xi, x)\big) - \sin \big(\zp B^{-n}_2(\eta, x)\big)\big]
				\\
				&=& \sum_{k=1}^\infty \kappa^k \big[\sin \big(\zp B^{k}_2(\xi, x)\big) - \sin \big(\zp B^{k}_2(\eta, x)\big)\big]\\
				&=& S(\xi, x) - S(\eta, x),
				\eea
				where we used that the first sum for non-negative integers $n$ is zero.
				
				This completes the proof. \epf

The following result describes the scaling properties of $G$.

\begin{lem}
\label{l:scaling_G}
For $\xi, x\in[0,1]$ we have
$$G\big(B^{-1}(\xi, x)\big) = \kappa G(\xi,x).$$
\end{lem}

\pf
Note that by definition, defining $n+1 = k$, for $\xi, x\in[0,1]$
\bea
G\big(B^{-1}(\xi, x)\big)
&=& \zp\, \sum_{n\in\Z} \kappa^{-n} \big[\sin \big(\zp B^{-n-1}(\xi, x)\big) - \sin \big(\zp B^{-n-1}(0, x)\big) \big]\\
&=& \zp \kappa \sum_{k\in\Z} \kappa^{-k} \big[\sin \big(\zp B^{-k}(\xi, x)\big) - \sin \big(\zp B^{-k}(0, x)\big)\big] \\
&=& \kappa G(\xi, x).
\eea
This is the claimed identity.\epf

We finish this section by giving a representation of $G$ which will be the starting point for our subsequent approach of transversality of $S$.
To this end, fix $\xi \in [0,1]$. We recursively define the following sequence of upward jumps in the dyadic sequence associated with $\xi$. For $n\in \N$ let

\begin{align}
\tau_{1}=\inf\{\ell\geq 0:\bar{\xi}_{-\ell}=1\},
\text{ and }
\tau_{n+1}=\inf\{\ell>\tau_n:\bar{\xi}_{-\ell}=1\},
\end{align}

and for $x\in[0,1]$
\begin{align*}\label{defg1}
g(x):=&2\pi\sum^{\infty}_{m=0}\kappa^m\Big[
\sin \Big(2\pi B_2^{m}(0,\frac{1+x}{2})\Big)-\sin \Big(2\pi B_2^{m}(0,\frac{x}{2})\Big)\Big]\\
=&2\pi\sum^{\infty}_{m=0}\kappa^m\Big[
\sin \Big( 2\pi\frac{1+x}{2^{m+1}}\Big)-\sin \Big(2\pi\frac{x}{2^{m+1}}\Big)\Big]\\
=&4\pi\sum^{\infty}_{m=0}\kappa^{m}
	\sin \Big( \frac{\pi}{2^{m+1}}\Big)\cos \Big(\pi \frac{2x+1}{2^{m+1}}\Big).
	\end{align*}

We have the following result.

\begin{pr}\label{p:representation_S}
Let $\xi,x\in[0,1]$. Then
\be\label{e:representation_G}
G(\xi,x)=\sum^{\infty}_{\ell=1}\kappa^{\tau_\ell+1}g\big(B_2^{\tau_\ell}(\xi,x)\big).
\ee
\end{pr}

\pf
It follows from the definition of $\tau$ that $\xi$ can be written
$$
\xi=(0,\ldots,0,\underbrace{1}_{\tau_{-1}},0,\ldots,0,\underbrace{1}_{\tau_{-2}},\ldots).
$$
For $n\in \N$ let
$$
\xi^n=(\bar{\xi}_0,\ldots,\bar{\xi}_{-\tau_n},0,0\ldots).
$$
We have
$$
G(\xi,x) = S(\xi,x)-S(0,x)=\sum^{\infty}_{\ell=1}\Big(S(\xi^\ell,x)-S(\xi^{\ell-1},x)\Big),
$$
where $\xi^0=(0,0,\ldots).$

For $\ell\in \N$ let us calculate $S(\xi^\ell,\cdot)-S(\xi^{\ell-1},\cdot)$. Since $\xi^{\ell}_{-k}=\xi^{\ell-1}_{-k}$ for $k\leq \tau_\ell-1$, we have
\begin{align}
S(\xi^\ell,x)-S(\xi^{\ell-1},x)
=&2\pi\sum^{\infty}_{n=1}\kappa^n\Big[
\sin \big(2\pi B_2^n(\xi^{\ell},x)\big)-\sin \big(2\pi B_2^n(\xi^{\ell-1},x)\big)\Big]\notag
\\
=&2\pi\sum^{\infty}_{n=\tau_\ell +1}\kappa^n\Big[
\sin \big(2\pi B_2^n(\xi^{\ell},x)\big)-\sin \big(2\pi B_2^n(\xi^{\ell-1},x)\big)\Big]\notag\\
=&2\pi\kappa^{\tau_\ell}\sum^{\infty}_{m=1}\kappa^m\Big[
\sin \Big(2\pi B_2^m\big(B_2^{\tau_\ell}(\xi^{\ell},x)\big)\Big)
-\sin \Big(2\pi B_2^m\big(B_2^{\tau_\ell}(\xi^{\ell-1},x)\big)\Big)\Big].\notag
\end{align}
Now
$$
B_2^{\tau_\ell}(\xi^{\ell},x)=B_2^{\tau_\ell}(\xi,x)=B_2^{\tau_\ell}(\xi^{\ell-1},x),
$$
and
$$
B_2^1\big(B_2^{\tau_\ell}(\xi^{\ell},x)\big)
=B_2^{\tau_\ell+1}(\xi^{\ell},x)
=\frac{1+B_2^{\tau_\ell}(\xi^{\ell},x)}{2}
=\frac{1+B_2^{\tau_\ell}(\xi,x)}{2},
$$
while
$$
B_2^1\textbf{}\big(B_2^{\tau_\ell}(\xi^{\ell-1},x)\textbf{}\big)
=B_2^{\tau_\ell+1}(\xi^{\ell},x)
=\frac{B_2^{\tau_\ell}(\xi,x)}{2}.
$$
So we may write by definition of $g$
\begin{align*}
S(\xi^\ell,x)-&S(\xi^{\ell-1},x)
\\
=
&2\pi\kappa^{\tau_\ell} \sum^{\infty}_{m=1} \kappa^m
\Big[
\sin \Big(2\pi B_2^{m-1}\Big(\frac{1+B_2^{\tau_\ell}(\xi,x)}{2}\Big)\Big)
-\sin \Big( 2\pi B_2^{m-1}\Big(\frac{B_2^{\tau_\ell}(\xi,x)}{2}\Big)\Big)\Big]\notag
\\
=& \kappa^{\tau_\ell+1}g\big(B_2^{\tau_\ell}(\xi,x)\big).
\end{align*}

Hence we obtain the claimed representation
$$
G(\xi,x)=\sum^{\infty}_{\ell=1}\kappa^{\tau_\ell+1} g\big(B_2^{\tau_\ell}(\xi,x)\big), \,\,\,\xi,x\in [0,1].
$$\epf

As a consequence of the representation of Proposition \ref{p:representation_S} we obtain the following representation that will be used in our analysis of transversality.

\begin{co}\label{c:representation_S}
Let $\xi,\eta\in[0,1]$ such that $\eh<\xi-\eta.$ Let $(\tau_n)_{n\ge 1}$ resp.~$(\sigma_n)_{n\ge 1}$ be the sequences of upward jumps in the dyadic sequences representing $\xi$ resp.~$\eta.$ Then $\tau_1=0, \tau_2\ge \sigma_1,$ and for $x\in[0,1]$
\begin{align*}
S(\xi,x)-S(\eta,x)
=&G(\xi,x)-G(\eta,x)\notag
\\
=&\kappa g(x)
+\sum^{\infty}_{\ell=2}\kappa^{\tau_\ell+1} g\big(B_2^{\tau_\ell}(\xi,x)\big)
-\sum^{\infty}_{m=1}\kappa^{\sigma_m+1}g\big(B_2^{\sigma_m}(\eta,x)\big)\notag
\\
=&\kappa g(x)+\kappa^{\tau_2+1}g(B_2^{\tau_2}(\xi,x))-\kappa^{\sigma_1+1}g\big(B_2^{\sigma_1}(\eta,x)\big)
\\
&\qquad \qquad
+\sum^{\infty}_{\ell=3}\kappa^{\tau_\ell+1}g\big(B_2^{\tau_\ell}(\xi,x)\big)
-\sum^{\infty}_{m=2}\kappa^{\sigma_m+1}g\big(B_2^{\sigma_m}(\eta,x)\big).
\end{align*}
\end{co}

\pf
It is clear from $\eh<\xi-\eta$ that $\tau_1=0,$ and $\tau_2\le \sigma_1.$ Therefore, the claimed formula readily follows from Proposition \ref{p:representation_S}.\epf

Now observe that for the upward jump representations of $\xi, \eta\in[0,1]$ such that $\eh<\xi-\eta$ we may also write
\begin{align*}
B_2^{\tau_2}(\xi,x)=&\frac{x}{2^{\tau_2}}+\frac{\bar{\xi}_0}{2^{\tau_2}}+\ldots+\frac{\bar{\xi}_{-\tau_1}}{2^{\tau_2-\tau_1}}=\frac{x}{2^{\tau_2}}+\frac{1}{2^{\tau_2-\tau_1}},\\
B_2^{\sigma_1}(\xi,x)=&\frac{x}{2^{\sigma_1}},\\
B_2^{\tau_\ell}(\xi,x)=&\frac{x}{2^{\tau_\ell}}+\frac{\bar{\xi}_0}{2^{\tau_2}}+\ldots+\frac{\bar{\xi}_{-\tau_1}}{2^{\tau_\ell-\tau_1}}+\ldots+\frac{\bar{\xi}_{-\tau_2}}{2^{\tau_\ell-\tau_2}}+\ldots+\frac{\bar{\xi}_{-\tau_\ell-1}}{2^{\tau_\ell-\tau_{\ell-1}}}\\
=&\frac{x}{2^{\tau_\ell}}+\frac{1}{2^{\tau_\ell-\tau_1}}+\ldots+\frac{1}{2^{\tau_\ell-\tau_{\ell-1}}},\\
B_2^{\sigma_m}(\xi,x)=&\frac{x}{2^{\sigma_m}}+\frac{1}{2^{\sigma_m-\sigma_1}}+\ldots+\frac{1}{2^{\sigma_m-\sigma_{m-1}}}.
\end{align*}  These equations allow to translate the representation of Corollary \ref{c:representation_S} into the following formula.

\begin{co}\label{c:representation_S2}
For $\xi, \eta\in[0,1]$ such that $\eh<\xi-\eta$ we have
\begin{align}
S(\xi,x)-S(\eta,x)
=&\kappa g(x)+\kappa^{\tau_2+1}g\Big(\frac{x}{2^{\tau_2}}+\frac{1}{2^{\tau_2-\tau_1}}\Big)-\kappa^{\sigma_1+1}g\Big(\frac{x}{2^{\sigma_1}}\Big)\notag\\
&\quad +\sum^{\infty}_{\ell=3}\kappa^{\tau_\ell+1}g\Big(\frac{x}{2^{\tau_\ell}}+\frac{1}{2^{\tau_\ell-\tau_1}}+\ldots+\frac{1}{2^{\tau_\ell-\tau_{\ell-1}}}\Big)\notag\\
&\qquad -\sum^{\infty}_{m=2}\kappa^{\sigma_m+1}g\Big(\frac{x}{2^{\sigma_m}}+\frac{1}{2^{\sigma_m-\sigma_1}}+\ldots+\frac{1}{2^{\sigma_m-\sigma_{m-1}}}\Big).
\end{align}
\end{co}

				\section{Transversality of $S$}\label{s:transversality}
In this section we shall establish a property of $S$ which will turn out to be crucial for the smoothness of the SBR measure deduced subsequently. It is called \emph{transversality}. We say that $S$ is \emph{transversal} if the vector
$$V(x) = \big(\,S(\xi,x)-S(\eta,x), S'(\xi,x)-S'(\eta,x)\,\big)$$
cannot be zero, for any $x\in[0,1],$ on the set $\{\eh<|\xi-\eta|\}$. This notion paraphrases a property of the flow described by $S(\xi,\cdot)-S(\eta,\cdot)$, and states that local extrema of $S(\xi,\cdot)-S(\eta,\cdot)$ cannot lie on the $x$-axis, i.e.~that there is motion transversal to the flow lines. We will design an interval $I$ in $]\eh,1[$ such that for $\kappa\in I$ the map $S$ is transversal. To deduce the property, we shall employ the representation obtained in Corollaries \ref{c:representation_S} and \ref{c:representation_S2} which mainly depends on the base function $g$. So we first discuss the geometric properties of $g$ in the following subsection.

\subsection{Geometric properties of $g$}\label{ss:geometry_g}

Recall that for $x\in[0,1]$
$$g(x) = 4\pi\sum^{\infty}_{m=0}\kappa^{m}
	\sin \Big( \frac{\pi}{2^{m+1}}\Big)\cos\Big( \pi \frac{2x+1}{2^{m+1}}\Big).$$
So its first and second derivatives $g'$ and $g''$ are given by
\begin{align}
g'(x)=&-4\pi^2\sum^{\infty}_{m=0}\Big(\frac{\kappa}{2}\Big)^{m}
\sin \Big( \frac{\pi}{2^{m+1}}\Big)\sin \Big(\pi \frac{2x+1}{2^{m+1}}\Big),\label{defgpri1}
\\
g''(x)=&-4\pi^3\sum^{\infty}_{m=0}\Big(\frac{\kappa}{4}\Big)^{m}
\sin \Big(\frac{\pi}{2^{m+1}}\Big)\cos \Big(\pi \frac{2x+1}{2^{m+1}}\Big). \label{defgdpri1}
\end{align}

By exploiting the properties of $g, g',$ and $g''$ we shall show that $g$ is a convex function up to a small interval $[0,x_0]$ near zero, with $x_0\le .027$ for any $\kappa\le \kappa_0$, where $\kappa_0\in[.55,.56]$, on which it is strictly decreasing. To show this we shall first do an analysis of $g''$ and show that it is positive except on $[0,.027]$. In our arguments, the functions are mostly approximated by the first two terms of their series expansions, with appropriate error estimates.

\begin{lem}\label{l:positivity_gdoubleprime}
There exists $0<x_0<.027$ such that for any $\kappa\le .6$ we have $g''|_{[x_0,1]}>0.$
\end{lem}

\pf
Approximating by the first two terms, we get for any $\kappa\ge \eh, x\in[0,1]$
\bea
g''(x)
&=&-4\pi^3\Big[\cos\Big(\frac{\pi}{2}(1+2x)\Big)+\frac{\kappa}{4}\sin\Big(\frac{\pi}{4}\Big)\cos\Big(\frac{\pi}{4}(1+2x)\Big)\Big]+R(\kappa,x)
\\
&=& -4\pi^3\Big[-2\sin\Big(\frac{\pi}{2}x\Big)\cos\Big(\frac{\pi}{2}x\Big)+\frac{\kappa}{8}\Big(\cos\Big(\frac{\pi}{2}x\Big)-\sin\Big(\frac{\pi}{2}x\Big)\Big)\Big]+R(\kappa,x),
\eea
where
$$
R(\kappa,x)=-4\pi^3\sum^{\infty}_{m=2}\Big(\frac{\kappa}{4}\Big)^{m}
\sin \Big(\frac{\pi}{2^{m+1}}\Big)\cos \Big(\pi \frac{2x+1}{2^{m+1}}\Big).
$$
Let us estimate $R$. Since $\cos$ is decreasing on $[0,\pi]$, we have
$$
R(\kappa,0)\leq R(\kappa,x)\leq R(\kappa,1).
$$
We estimate $R(\kappa,0)$ below and  $R(\kappa,1)$ above. We have

 \begin{align*}
 R(\kappa,0)=&-4\pi^3\sum^{\infty}_{m=2}\Big(\frac{\kappa}{4}\Big)^{m}
 \sin \Big(\frac{\pi}{2^{m+1}}\Big)\cos \Big( \frac{\pi}{2^{m+1}}\Big)
=-2\pi^3\sum^{\infty}_{m=2}\Big(\frac{\kappa}{4}\Big)^{m}\sin \Big(\frac{\pi}{2^{m}}\Big).
 \end{align*}
 Now for $m\geq 2$
 $$
 \frac{ \sin \big(\frac{\pi}{2^{m}}\big)}{\frac{\pi}{2^{m}}}\leq 1 ,\,\,\, \text {hence } \,\,\,\sin \Big(\frac{\pi}{2^{m}}\Big)\leq \frac{\pi}{2^{m}},
 $$
 and therefore
 \begin{align*}
 R(\kappa,0)\geq&-2\pi^3\sum^{\infty}_{m=2}\Big(\frac{\kappa}{4}\Big)^{m}
  \frac{\pi}{2^{m}}
  =-2\pi^4\sum^{\infty}_{m=2}\Big(\frac{\kappa}{8}\Big)^{m}\\
  =&-2\pi^4\Big(\frac{\kappa}{8}\Big)^2\cdot\frac{1}{1-\frac{\kappa}{8}}=-\frac{\pi^4}{4}\frac{\kappa^2}{8-\kappa}
  .
 \end{align*}

Moreover

  \begin{align*}
 R(\kappa,1)=&-4\pi^3\sum^{\infty}_{m=2}\Big(\frac{\kappa}{4}\Big)^{m}
 \sin \Big(\frac{\pi}{2^{m+1}}\Big)\cos\Big(  \frac{3\pi}{2^{m+1}}\Big)
\\
 =&-2\pi^3\sum^{\infty}_{m=2}\Big(\frac{\kappa}{4}\Big)^{m}
 \Big(\sin \Big(\frac{\pi}{2^{m-1}}\Big)-\sin \Big(\frac{\pi}{2^{m}}\Big)\Big)
\\
 =&-2\pi^3\sum^{\infty}_{m=2}\Big(\frac{\kappa}{4}\Big)^{m}
\sin \Big(\frac{\pi}{2^{m-1}}\Big)
+2\pi^3\sum^{\infty}_{m=2}\Big(\frac{\kappa}{4}\Big)^{m}
\sin \Big(\frac{\pi}{2^{m}}\Big).
 \end{align*}
Let us estimate both terms separately. Firstly,
\begin{align*}
 2\pi^3\sum^{\infty}_{m=2}\Big(\frac{\kappa}{4}\Big)^{m}
\sin \Big(\frac{\pi}{2^{m}}\Big) \leq
& 2\pi^3\sum^{\infty}_{m=2}\Big(\frac{\kappa}{4}\Big)^{m}
 \frac{\pi}{2^{m}}
 =\frac{\pi^4}{4}\frac{\kappa^2}{8-\kappa}
\end{align*}
 Secondly, for $m\geq 2$
 \begin{align*}
 \sin \Big(\frac{\pi}{2^{m-1}}\Big)
=
&\frac{\sin \big(\frac{\pi}{2^{m-1}}\big)}{\frac{\pi}{2^{m-1}}}\frac{\pi}{2^{m-1}}
\\
 \geq
& \frac{\sin \big(\frac{\pi}{2}\big)}{\frac{\pi}{2}}\frac{\pi}{2^{m-1}}=\frac{2}{\pi}\cdot\frac{\pi}{2^{m-1}}=\frac{1}{2^{m-2}}.
 \end{align*}
Therefore
 \begin{align*}
 -2\pi^3\sum^{\infty}_{m=2}\Big(\frac{\kappa}{4}\Big)^{m}
 \sin \Big(\frac{\pi}{2^{m-1}}\Big)
\leq
& -2\pi^3\sum^{\infty}_{m=2}\Big(\frac{\kappa}{4}\Big)^{m}
 \frac{\pi}{2^{m-1}} \notag
\\
 =&-8\pi^3\sum^{\infty}_{m=2}\Big(\frac{\kappa}{8}\Big)^{m}=-8\pi^3\Big(\frac{\kappa}{8}\Big)^{2}\frac{8}{8-\kappa}
=-\pi^3\frac{\kappa^2}{8-\kappa}.
 \end{align*}
In summary we obtain
 \begin{align*}
 R(\kappa,1)\leq& -\pi^3\frac{\kappa^2}{8-\kappa}+\frac{\pi^4}{4}\frac{\kappa^2}{8-\kappa}=-\frac{\pi^4}{4}\frac{\kappa^2}{8-\kappa}\Big(-1+\frac{4}{\pi}\Big).
 \end{align*}
The desired estimate for the remainder term thus reads
 $$-\frac{\pi^4}{4}\frac{\kappa^2}{8-\kappa}\leq R(\kappa,x)\leq -\frac{\pi^4}{4}\frac{\kappa^2}{8-\kappa}\Big(\frac{4}{\pi}-1\Big)\leq -\frac{\pi^4}{4}\frac{\kappa^2}{8-\kappa}\cdot 0.27.
 $$
 Therefore, we also obtain that the root of $g''$ is located between the unique roots of
 $$
 \bar h (x)=-2\sin \Big(\frac{\pi}{2}x\Big)\cos\Big(\frac{\pi}{2}x\Big)
+\frac{\kappa}{8}\Big(\cos\Big(\frac{\pi}{2}x\Big)-\sin\Big(\frac{\pi}{2}x\Big)\Big)+\frac{\pi}{16}\frac{\kappa^2}{8-\kappa}
 $$
 and
  $$
 \underline {h} (x)
=-2\sin \Big(\frac{\pi}{2}x\Big)\cos\Big(\frac{\pi}{2}x\Big)
+\frac{\kappa}{8}\Big(\cos\Big(\frac{\pi}{2}x\Big)-\sin\Big(\frac{\pi}{2}x\Big)\Big)+\frac{\pi}{16}\cdot 0.27\cdot\frac{\kappa^2}{8-\kappa}
 $$
Clearly, $\underline {h}(0)>0 $, and we see numerically that $\underline {h}(.027)<0$ for all $\kappa\le \kappa_0.$
Hence the root of $g''$ is bounded above by $.027$. See Figure \ref{Plotg''}.\epf

Note that in many of the subsequent figures upper (in red) and lower (in blue) estimates are very close, so that in print they appear as practically indistinguishable.


\begin{figure}[ht]
\centering
\includegraphics[width=\textwidth]{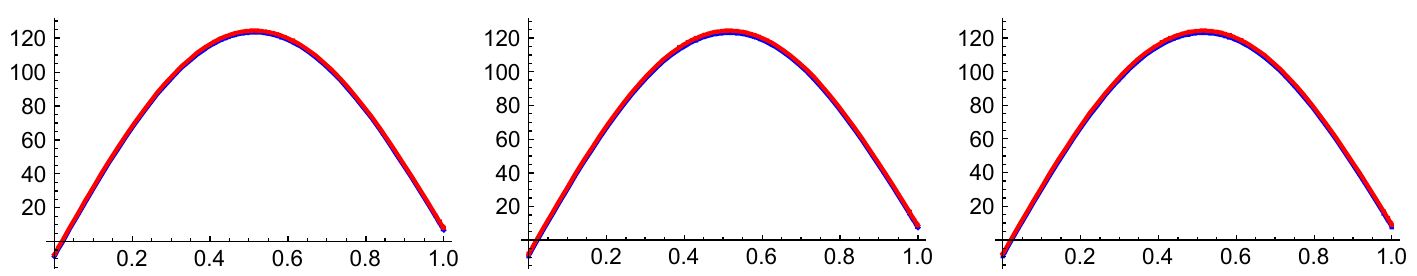}
\caption{Upper and lower approx.~of graph of $g''$ for $\kappa=0.5, 0.55, 0.56$ (left to right). The upper and lower approx.~are very close and nearly indistinguishable, as in several graphics of the paper. For a certain $x_0$, $g''|_{[x_0,1]}>0$.}
\label{Plotg''}
\end{figure}

We next discuss geometric properties of $g'$.

\begin{lem}\label{l:geometry_gprime}
There exist $0<y_0<y_1$ such that $y_0\ge .55, y_1\le .6$ and such that for any $\kappa\le .6$ the unique root of $g'$ is located in $[y_0,y_1].$
$g'$ is negative below this root, and positive above.
\end{lem}

\pf

We have for $x\in[0,1]$
\begin{align*}
g'(x)=
&-4\pi^2\Big(\cos (\pi x)+\frac{\kappa}{2}\sin\Big(\frac{\pi}{4}\Big)\sin\Big(\frac{\pi(1+2x)}{2}\Big)\Big)
\\
&\qquad \qquad \qquad
-4\pi^2\sum^{\infty}_{m=2}\Big(\frac{\kappa}{4}\Big)^{m}
\sin \Big(\frac{\pi}{2^{m+1}}\Big)\sin \Big(\pi \frac{2x+1}{2^{m+1}}\Big)
\\
=&-4\pi^2\Big[\cos^2\Big(\frac{ \pi x}{2}\Big)
             -\sin^2\Big(\frac{ \pi x}{2}\Big)
						 +\frac{\kappa}{4}\Big(\cos\Big(\frac{ \pi x}{2}\Big)
						                       +\sin\Big(\frac{ \pi x}{2}\Big)\Big)\Big]+ R(\kappa,x).
\end{align*}
Let us estimate $R$. We have
$$
R(\kappa,1)\leq R(\kappa,x)\leq R(\kappa,0).
$$
Moreover
\begin{align*}
R(\kappa,0)
\leq
&-4\pi^2\sum^{\infty}_{m=2}\Big(\frac{\kappa}{2}\Big)^{m}
\Big(\frac{\sin \big(\frac{\pi}{8}\big)}{\frac{\pi}{8}}\Big)^2
\Big(\frac{\pi}{2^{m+1}}\Big)^2
=
-64\pi^2\sum^{\infty}_{m=2}\Big(\frac{\kappa}{8}\Big)^{m}\sin^2\Big(\frac{\pi}{8}\Big)
\notag
\\
\le
&-64\pi^2\cdot 0.14\cdot \frac{\kappa^2}{8^2}\frac{8}{8-\kappa}
=
-8\pi^2\cdot 0.14\cdot\frac{\kappa^2}{8-\kappa},
\end{align*}
and
\begin{align*}
R(\kappa,1)=
&-4\pi^2\sum^{\infty}_{m=2}\Big(\frac{\kappa}{2}\Big)^{m}
\sin \Big(\frac{\pi}{2^{m+1}}\Big)\sin \Big( \frac{3\pi}{2^{m+1}}\Big)
\\
\geq&-4\pi^2\sum^{\infty}_{m=2}\Big(\frac{\kappa}{2}\Big)^{m}
 \frac{\pi}{2^{m+1}}  \frac{3\pi}{2^{m+1}}
 =
-3\pi^4\sum^{\infty}_{m=2}\Big(\frac{\kappa}{8}\Big)^{m}
=
-\frac{3\pi^4}{8}\frac{\kappa^2}{8-\kappa}.
\end{align*}
So we get
$$
-\frac{3\pi^4}{8}\frac{\kappa^2}{8-\kappa}\leq R(\kappa,x)\leq-  8\pi^2\cdot 0.14\cdot\frac{\kappa^2}{8-\kappa}.
$$
Correspondingly, the root of $g'$ is comprised between the roots of
\begin{align*}
\underline{h}(x)=
&\cos^2\Big(\frac{ \pi x}{2}\Big)
-\sin^2\Big(\frac{ \pi x}{2}\Big)
+\frac{\kappa}{4}\Big(\cos\Big(\frac{ \pi x}{2}\Big) +\sin\Big(\frac{ \pi x}{2}\Big)\Big)+\underbrace{\frac{3\pi^2}{32}}_{\equiv 0.92}\frac{\kappa^2}{8-\kappa}
\end{align*}
and
\begin{align*}
\bar{h}(x)=
&\cos^2\Big(\frac{ \pi x}{2}\Big)
 -\sin^2\Big(\frac{ \pi x}{2}\Big)
 +\frac{\kappa}{4}\Big(\cos\Big(\frac{ \pi x}{2}\Big) +\sin\Big(\frac{ \pi x}{2}\Big)\Big)+0.28\frac{\kappa^2}{8-\kappa}.
\end{align*}
A numerical estimate easily yields $0<y_0<y_1$, $y_0\ge .55$ and $y_1\le .6$ such that the unique root of $g'$ is located in $[y_0,y_1]$, see Figure \ref{Plotg'}. For uniqueness, an appeal to Lemma \ref{l:positivity_gdoubleprime} is enough. \epf

\begin{figure}[ht]
\centering
\includegraphics[width=\textwidth]{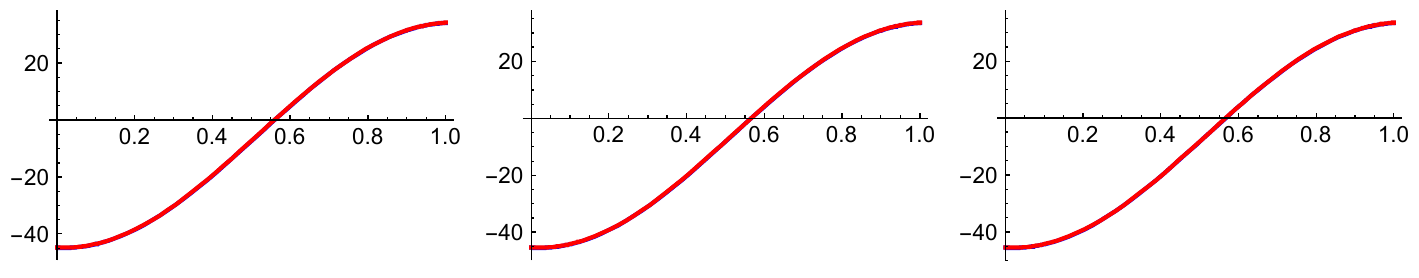}
\caption{Upper and lower approx.~of graph of $g'$ for $\kappa=0.5, 0.55, 0.56$ (left to right). Illustration of the geometric properties of $g'$ as per Lemma \ref{l:geometry_gprime}.}
\label{Plotg'}
\end{figure}

We finally assess the geometric properties of $g$.
\begin{lem}\label{l:geometry_g}
There exist $0<z_0<z_1$ such that $z_0\ge .05, y_1\le .15$ and such that for any $\kappa\le .6$ the unique root of $g$ is located in $[z_0,z_1].$
$g$ is strictly convex on $[x_0,1]$ and strictly decreasing on $[0,x_0].$ It has a strict global minimum at the root of $g'$, i.e.~in $[y_0,y_1].$
\end{lem}

\pf
An appeal to Lemma \ref{l:geometry_gprime} and Lemma \ref{l:positivity_gdoubleprime} shows that it only remains to estimate the location of the root of $g$.
We have
\bea
g(x)
&=&
4\pi\Big[\cos\Big(\frac{\pi}{2}(1+2x)\Big)
  +\kappa\sin\Big(\frac{\pi}{4}\Big)\cos\Big(\frac{\pi}{4}(1+2x)\Big)\Big] +R(\kappa,x)
	\\
&=&
4\pi\Big[-2\sin\Big(\frac{\pi x}{2}\Big)\cos\Big(\frac{\pi x}{2}\Big)
+\frac{\kappa}{2}\Big(\cos\Big(\frac{\pi x}{2}\Big)-\sin\Big(\frac{\pi x}{2}\Big)\Big)\Big] +R(\kappa,x),
\eea
with
\begin{align*}
R(\kappa,x)=&4\pi\sum^{\infty}_{m=2}\kappa^{m}
\sin \Big( \frac{\pi}{2^{m+1}}\Big)\cos\Big( \pi \frac{2x+1}{2^{m+1}}).
\end{align*}
By monotonicity of $\cos$ on $[0,\pi]$ we have
$$
R(\kappa,1)\leq  R(\kappa,x)\leq  R(\kappa,0).
$$
We estimate $R(\kappa,1)$ below and $R(\kappa,0)$ above. We may write
\begin{align*}
R(\kappa,0)=&2\pi\sum^{\infty}_{m=2}\kappa^{m}
\sin \Big( \frac{\pi}{2^{m}}\Big)
\\
\leq&2\pi^2\sum^{\infty}_{m=2}\Big(\frac{\kappa}{2}\Big)^{m}
=2\pi^2\Big(\frac{\kappa}{2}\Big)^{2}\frac{2}{2-\kappa}
=\pi^2\frac{\kappa^2}{2-\kappa},
\end{align*}
and
\begin{align*}
R(\kappa,1)
=&
2\pi\sum^{\infty}_{m=2}\kappa^{m}
\Big[\sin \Big( \frac{\pi}{2^{m-1}}\Big)
    -\sin  \Big(\frac{\pi}{2^{m}}\Big)\Big]
=
2\pi\Big[\kappa\sum^{\infty}_{m=1}\kappa^{m}
\sin \Big( \frac{\pi}{2^{m}}\Big)
-\sum^{\infty}_{m=2}\kappa^{m}
\sin  \Big(\frac{\pi}{2^{m}}\Big)\Big]
\\
=&2\pi\kappa^2 +2\pi(\kappa-1) \sum^{\infty}_{m=2}\kappa^{m}\sin \Big( \frac{\pi}{2^{m}}\Big)
\\
\geq
&
 2\pi\kappa^2 +2\pi(\kappa-1) \sum^{\infty}_{m=2}\kappa^{m} \frac{\pi}{2^{m}}
\\
=&2\pi\kappa^2 +2\pi^2(\kappa-1) \frac{\kappa^2}{4}\frac{2}{2-\kappa}
=\frac{\pi\kappa^2}{2-\kappa}(4-\pi+\kappa(\pi-2)).
\end{align*}
Hence we have
$$
\frac{\pi\kappa^2}{2-\kappa}\big(4-\pi+\kappa(\pi-2)\big)
\leq
R(\kappa,x)\leq\frac{\pi\kappa^2}{2-\kappa}\pi^2,
$$
and the root of $g$ is comprised between the roots of
$$
\underline{h}(x)
=-2\sin\Big(\frac{\pi x}{2}\Big)\cos\Big(\frac{\pi x}{2}\Big)
+\frac{\kappa}{2}\Big(\cos\Big(\frac{\pi x}{2}\Big)-\sin\Big(\frac{\pi x}{2}\Big)\Big)
+\frac{\kappa^2}{2(2-\kappa)}\big(4-\pi+\kappa(\pi-2)\big)
$$
and
$$
\bar h(x)
=-2\sin\Big(\frac{\pi x}{2}\Big)
   \cos\Big(\frac{\pi x}{2}\Big)
	+\frac{\kappa}{2}\Big(\cos\Big(\frac{\pi x}{2}\Big)-\sin\Big(\frac{\pi x}{2}\Big)\Big)+\frac{\kappa^2}{2(2-\kappa)}.
$$
A numerical estimate of the roots of $\underline{h}$ and $\bar{h}$ yields $z_0, z_1\in[0,1]$ with the desired properties. See Figure \ref{Plotg}. \epf

\begin{figure}[ht]
\centering
\includegraphics[width=\textwidth]{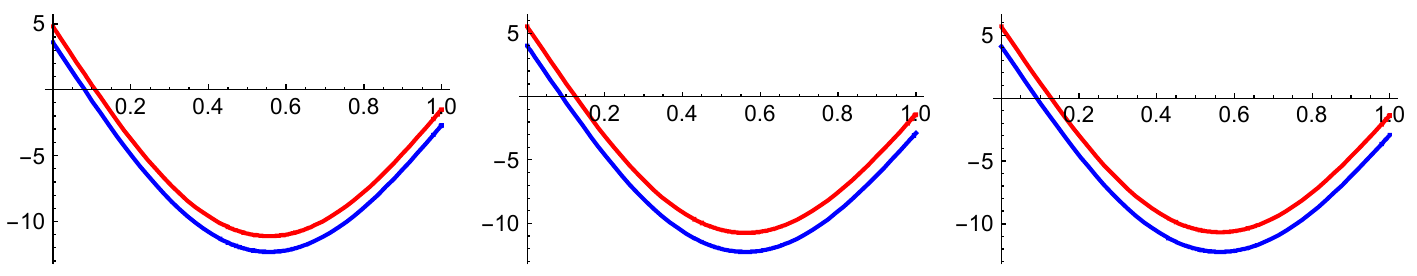}
\caption{Upper and lower approx.~of graph of $g$ for $\kappa=0.5, 0.55, 0.56$ (left to right). Illustration of the geometric properties of $g$ as per Lemma \ref{l:geometry_g}.}
\label{Plotg}
\end{figure}

\subsection{Transversality via properties of $g$}\label{ss:transversality_via_g}

In this subsection we will derive the transversality properties of $S$, starting from the representation in Corollary \ref{c:representation_S2}, and using the geometric properties of $g$ discussed in the previous subsection. We shall prove the following proposition.

\begin{pr}\label{p:transversality}
There exists $\kappa_0\in [0.55,0.56]$ such that for $\kappa\leq \kappa_0$ we have
	$$
	\inf_{x\in[0,1]} |V(x)|>0,
	$$
	uniformly on $\{\eh<|\xi-\eta|\}$.
	
	\end{pr}

\pf

For the proof, we shall approximate $S(\xi,\cdot)-S(\eta,\cdot)$ by the first three terms in the representation of Corollary \ref{c:representation_S2}. For reasons of symmetry in $\xi, \eta$ we shall argue only on the set $\{\eh< \xi-\eta\}.$ We will denote the sequences marking the upward jumps in the dyadic sequence representing $\xi$ by $(\tau_n)_{n\in\N}$, and the one for $\eta$ by $(\sigma_n)_{n\in\N}$. Here we tacitly assume that we omit dyadic $\xi, \eta$ that constitute a set of measure zero on $[0,1].$ We recall that $\tau_1=0,$ and $\sigma_1\ge \tau_2$ on this set.
Taking these prerequisites into account, we have to deal with the approximation
\begin{align}\label{eqrepS1}
S(\xi,x)-S(\eta,x)
=&\kappa g(x)+\kappa^{\tau_2+1}g\Big(\frac{x}{2^{\tau_2}}+\frac{1}{2^{\tau_2-\tau_1}}\Big)-\kappa^{\sigma_1+1}g\Big(\frac{x}{2^{\sigma_1}}\Big)\notag
\\
&+\kappa^{\tau_3+1}g\Big(\frac{x}{2^{\tau_3}}+\frac{1}{2^{\tau_3}}+\frac{1}{2^{\tau_3-\tau_2}}\Big)-\kappa^{\sigma_2+1}g\Big(\frac{x}{2^{\sigma_2}}+\frac{1}{2^{\sigma_2-\sigma_1}}\Big) \notag
\\
&+O(\kappa^{\tau_4+1})+O(\kappa^{\sigma_3+1})\notag\\
=& f_{\eta,\xi}+O(\kappa^{\tau_4+1})+O(\kappa^{\sigma_3+1}),
\end{align}
for which the error terms will have to be estimated. The values of $\tau_2, \tau_3, \sigma_1, \sigma_2$ for the different possible values of $\bxi_{-1}, \bxi_{-2}, \beeta_{-1}, \beeta_{-2}$ are listed in Table \ref{table:Cases}.

\begin{table}[ht]%
	\centering
	\begin{tabular}{c|cccccccc}
		Cases & $\bxi_{-1}$ & $\beeta_{-1}$ & $\bxi_{-2}$ & $\beeta_{-2}$&$\tau_2$& $\tau_3$& $\sigma_1$& $\sigma_2$\\
		\hline
		1 & 0&0&0&0&$\geq$ 3& $\geq$ 4 & $\geq$ 3&$ \geq$ 4\\
		2 & 0&0&1&0&= 2& $\geq$ 3 & $\geq$ 3& $\geq$ 4\\
		3 & 0&0&1&1&= 2& $\geq$ 3 & =2& $\geq$ 3\\
		4 & 1&0&0&0&=1& $\geq $3 & $\geq$ 3& $\geq$ 4\\
		5 & 1&0&0&1&=1& $\geq$ 3 & =2& $\geq$ 3\\
		6 & 1&0&1&0&=1& =2 & $\geq$ 3& $\geq$ 4\\
		7 & 1&0&1&1&=1& =2 & =2& $\geq$ 3\\
		8 & 1&1&0&0&=1& $\geq$ 3 & =1& $\geq$ 3\\
		9 & 1&1&1&0&=1& =2 & =1& $\geq $3\\
		10 & 1&1&1&1&=1& =2 & =1& =2\\
		\hline
	\end{tabular}
	\caption{
		Possible cases for the components $\bxi_{-i}, \beeta_{-i}$, $i=1,2$ in the approximation of $S(\xi,\cdot)-S(\eta,\cdot)$ by its first three terms.}
	\label{table:Cases}
\end{table}

We also have

\begin{align}\label{e:eqrepSprime}
S'(\xi,x)-S'(\eta,x)
=
&
\kappa\Big\{ g'(x)+\big(\frac{\kappa}{2}\big)^{\tau_2}g'\big(\frac{x}{2^{\tau_2}}+\frac{1}{2^{\tau_2-\tau_1}}\big)-\big(\frac{\kappa}{2}\big)^{\sigma_1} g'\big(\frac{x}{2^{\sigma_1}}\big)
\notag\\
&+\big(\frac{\kappa}{2}\big)^{\tau_3} g'\big(\frac{x}{2^{\tau_3}}+\frac{1}{2^{\tau_3}}+\frac{1}{2^{\tau_3-\tau_2}}\big)-\big(\frac{\kappa}{2}\big)^{\sigma_2} g'\big(\frac{x}{2^{\sigma_2}}+\frac{1}{2^{\sigma_2-\sigma_1}}\big)
 \notag\\
&+O\Big(\big(\frac{\kappa}{2}\big)^{\tau_4}\Big)+O\Big(\big(\frac{\kappa}{2}\big)^{\sigma_3}\Big)\Big\}
\notag\\
=&\kappa\Big\{f'_{\xi,\eta}(x)
+O\Big(\big(\frac{\kappa}{2}\big)^{\tau_4}\Big)
+O\Big(\big(\frac{\kappa}{2}\big)^{\sigma_3}\Big)\Big\},
\end{align}
and

\begin{align}\label{eqrepSDprime}
S''(\xi,x)-S''(\eta,x)
=
&\kappa\Big\{ g''(x)+\big(\frac{\kappa}{4}\big)^{\tau_2}g''\big(\frac{x}{2^{\tau_2}}+\frac{1}{2^{\tau_2-\tau_1}}\big)-\big(\frac{\kappa}{4}\big)^{\sigma_1}g''\big(\frac{x}{2^{\sigma_1}}\big)
\notag
\\
&
+\big(\frac{\kappa}{4}\big)^{\tau_3}g''\big(\frac{x}{2^{\tau_3}}+\frac{1}{2^{\tau_3}}+\frac{1}{2^{\tau_3-\tau_2}}\big)-\big(\frac{\kappa}{4}\big)^{\sigma_2}g''\big(\frac{x}{2^{\sigma_2}}+\frac{1}{2^{\sigma_2-\sigma_1}}\big)
\notag
\\
&
+O\Big(\big(\frac{\kappa}{4}\big)^{\tau_4}\Big)+O\Big(\big(\frac{\kappa}{4}\big)^{\sigma_3}\Big)\Big\} \notag
\\
=&\kappa\Big\{f''_{\xi,\eta}(x)
+O\Big(\big(\frac{\kappa}{4}\big)^{\tau_4}\Big)
+O\Big(\big(\frac{\kappa}{4}\big)^{\sigma_3}\Big)\Big\}.
\end{align}

Our strategy of the proof of transversality mainly consists in establishing that on $\{\eh< \xi-\eta\}$ the function $S(\xi,\cdot)-S(\eta,\cdot)$ has similar properties of $g$ studied in the previous subsection, i.e.~it is convex on the interval $[.1,.9]$, decreasing on $[0,.1]$ and increasing on $[.9,1]$. It therefore possesses a unique global minimum on $[0,1].$ To prove $
\inf_{x\in[0,1]} |V(x)|>0$, it will be enough to show that the global minimum is below the axis. The essential part of this strategy is composed of the analysis of $S''(\xi,\cdot)-S''(\eta,\cdot)$, to obtain the convexity and monotonicity properties mentioned. These results are obtained in the auxiliary lemmas in the remainder of this section, from Lemma \ref{l:monotonicity_differences} to Lemma \ref{lemposSdpri2}, where we treat combinations of $g''$ arising at different arguments that appear in $f''_{\xi,\eta}$, and establish monotonicity properties.

As a consequence of Lemmas \ref{lemposSdpri} and \ref{lemposSdpri2} below we know that for (non-dyadic) $\xi, \eta\in[0,1]$ such that $\eh<\xi-\eta$ the function $S(\xi,\cdot)-S(\eta,\cdot)$ is strictly convex on $[.1,.9]$, strictly decreasing on $[0,.1]$ and strictly increasing on $[.9,1].$ Therefore it possesses a unique global minimum.

To complete the proof of transversality, we therefore just have to show that this minimum is not zero. We further know that $S(\xi,\cdot)-S(\eta,\cdot)$ is a perturbation of $g$, that has its global minimum on $[.55,.6]$. We finish the proof by giving numerical upper estimates of the values $S(\xi,.55)-S(\eta,.55)$, with a careful error estimate integrated, that are seen to be negative for $\kappa\le \kappa_0.$
The error is estimated by $2 \sum_{l=3}^\infty \kappa^{l+1} \|g\|_\infty \le 2\frac{\kappa^4}{1-\kappa} \|g\|_\infty$ (where $\|\cdot\|_\infty$ stands for the usual supremum norm). In Figure \ref{PlotS} we plot $S(\xi,.55)-S(\eta,.55)$ as a function of $\kappa$ on the interval $[\eh,.55]$ in the ten cases. This completes the proof of Proposition \ref{p:transversality}. \epf

\begin{figure}[h!tbp]
\centering
\includegraphics[width=\textwidth]{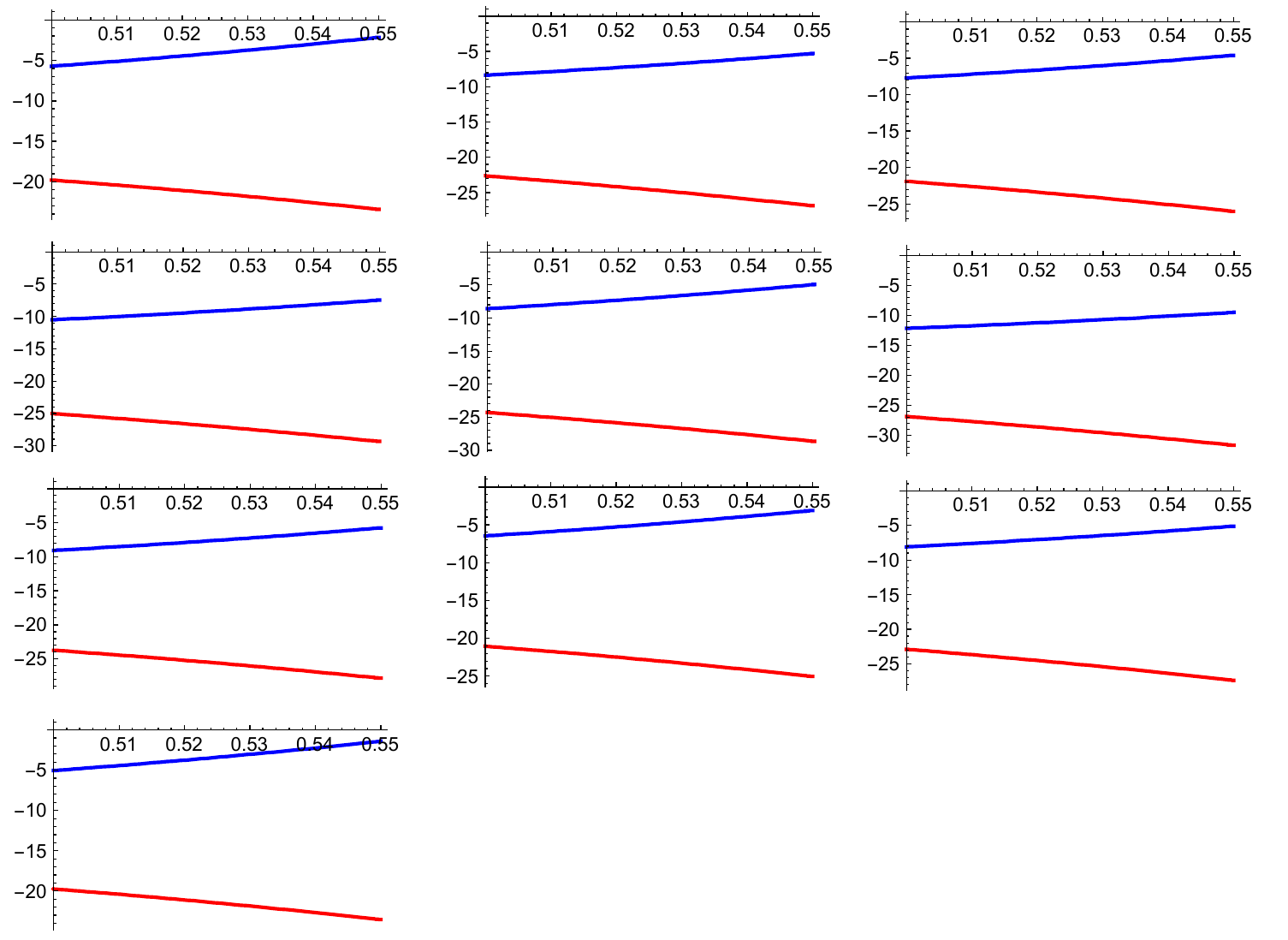}
\caption{Upper and lower approx.~of $S(\xi,.55)-S(\eta,.55)$ as function of $\kappa \in [\eh, .55]$, for the ten cases of Table \ref{table:Cases} (Case 1 on top left, continuing in the usual fashion left to right, top to bottom). In all ten cases, the difference is negative for the values of $\kappa$ as per conclusion of Proposition \ref{p:transversality}.}
\label{PlotS}
\end{figure}

\begin{lem}\label{l:monotonicity_differences}
We have
\begin{align*}
k_1(x):= g''(\frac{x+1}{2})-g''(\frac{x}{2})
=& 8\pi^3\sum^{\infty}_{m=0}(\frac{\kappa}{4})^{m}
\sin \Big( \frac{\pi}{2^{m+1}}\Big)
\sin \Big( \frac{\pi}{2^{m+2}}\Big)
\sin  \Big(\frac{\pi}{2^{m+2}}(3+2x)\Big),
\\
k_2(x) := g''(\frac{x+1}{4}+\frac{1}{2})-g''(\frac{x}{4}+\frac{1}{2})
=& 8\pi^3\sum^{\infty}_{m=0}(\frac{\kappa}{4})^{m}
\sin \Big( \frac{\pi}{2^{m+1}}\Big)
\sin \Big( \frac{\pi}{2^{m+3}}\Big)
\sin \Big( \frac{\pi}{2^{m+2}}(\frac{9}{2}+x)\Big),
\\
k_3(x) := g''(\frac{x+1}{4})-g''(\frac{x}{4})
=&8\pi^3\sum^{\infty}_{m=0}(\frac{\kappa}{4})^{m}
\sin \Big( \frac{\pi}{2^{m+1}}\Big)
\sin \Big(\frac{\pi}{2^{m+3}}\Big)
\sin \Big(\frac{\pi}{2^{m+2}}(3+x)\Big).
\end{align*}
Moreover, $k_1,\cdots,k_3$ are strictly decreasing on $[0,1]$, with $k_1(0),\cdots,k_3(0)>0$, and $k_1(1)<0$, whereas $k_2(1), k_3(1)>0$, such that in particular
$k_2, k_3$ are strictly positive on $[0,1].$
\end{lem}

\pf
The claimed equations follow readily from trigonometric identities of the form $\sin y- \sin z = 2 \sin \big(\frac{y-z}{2}\big) \cos \big(\frac{y+z}{2}\big), y,z\in[0,1].$ To deduce the monotonicity properties, we approximate the functions by their first term, and note that the remainder is small enough to not perturb its monotonicity properties.
For monotonicity properties of the first term, we observe that $\cos$ is negative on $[\frac{\pi}{2}, \frac{3 \pi}{2}],$ and the intervals $[\frac{3\pi}{4}, \frac{5\pi}{4}], [\frac{9\pi}{8},\frac{11\pi}{8}], [\frac{3\pi}{4},\pi]$ are contained in $[\frac{\pi}{2}, \frac{3 \pi}{2}].$ The signs of their values at $0$ and $1$ are obtained as well by evaluating the first terms. See Figure \ref{Plotk1234}. \epf

\begin{figure}[ht]
\centering
\includegraphics[width=\textwidth]{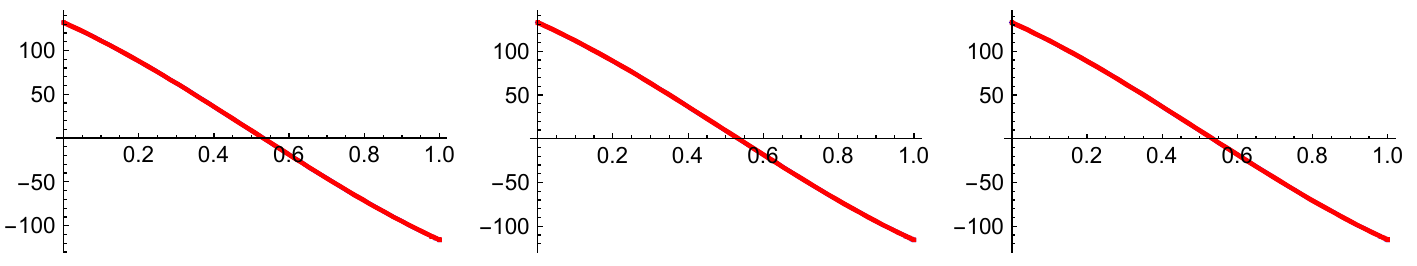}
\includegraphics[width=\textwidth]{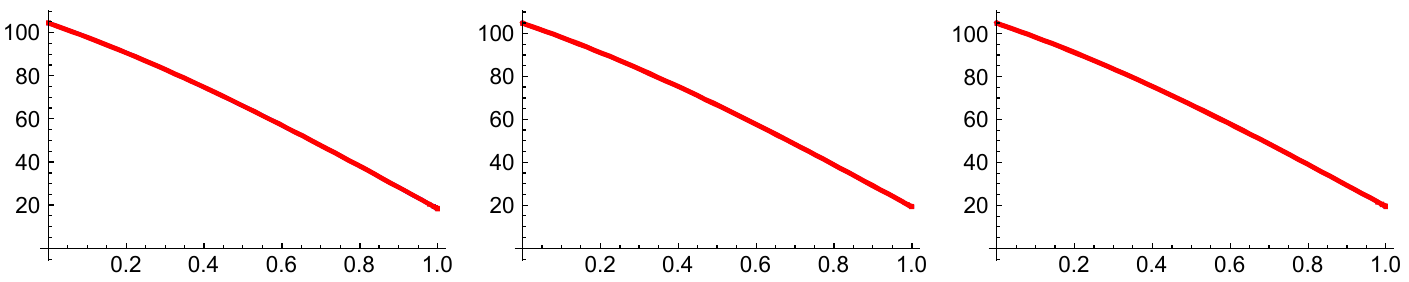}
\includegraphics[width=\textwidth]{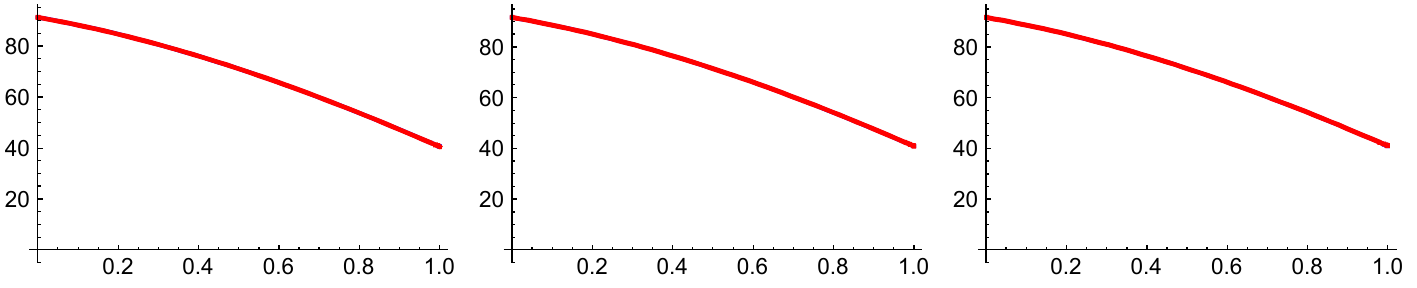}
\caption{Upper and lower approx.~of $k_1,k_2,k_3$ (top to bottom) for $\kappa=0.5, 0.55, 0.56$ (left to right). Properties according to Lemma \ref{l:monotonicity_differences}.}
\label{Plotk1234}
\end{figure}

A somewhat different case is treated in the following lemma.

\begin{lem}\label{l:monotonicity_differences2}
The function
$$
k_4(x):= g''\Big(\frac{x+1}{2}\Big)-\frac{\kappa}{4}g''\Big(\frac{x}{4}\Big),\quad x\in[0,1],
$$
is strictly decreasing and positive on $[0,.9]$.
\end{lem}

\pf
We estimate by taking the first two terms in the series expansion of $g''(\frac{x+1}{2})$ and the first term of $g''(\frac{\kappa}{4})$. We obtain
\begin{align*}
k_4(x)
=&-4\pi^3\Big( \cos\Big(\frac{\pi}{2}(2+x)\Big)-\frac{\kappa}{4}(1-\frac{\sqrt{2}}{2})\cos\Big(\frac{\pi}{4}(2+x)\Big)\Big)
+O\Big(\big(\frac{\kappa}{4}\big)^2\Big)
\\
=&-4\pi^3\Big(-\cos\Big(\frac{x\pi}{2}\Big)+\frac{\kappa}{4}(1-\frac{\sqrt{2}}{2})\sin\Big(\frac{x\pi}{4}\Big)\Big)
+O\Big(\big(\frac{\kappa}{4}\big)^2\Big)
\\
=&-4\pi^3\Big(-\cos^2\Big(\frac{x\pi}{4}\Big)+\sin^2\Big(\frac{x\pi}{4}\Big)+\frac{\kappa}{4}(1-\frac{\sqrt{2}}{2})\sin\Big(\frac{x\pi}{4}\Big)\Big)
+O\Big(\big(\frac{\kappa}{4}\big)^2\Big)
\end{align*}
Let $u=\sin\big(\frac{x\pi}{4}\big), 0\leq u\leq \frac{\sqrt{2}}{2}$. We have to discuss the sign of
$$
f(u)=1-2u^2-\frac{\kappa}{8}(2-\sqrt{2})u.
$$
But
$$
f'(u)=-4u-\frac{\kappa}{8}(2-\sqrt{2})<0.
$$
Hence $f$ in $u$ and thus the approximation of $k_4$ in $x$ is decreasing. It remains to take the value at $x=0.9$ to see that $k_4$ is positive on $[0,.9]$. See Figure \ref{Plotk5}. \epf

\begin{figure}[ht]
\centering
\includegraphics[width=\textwidth]{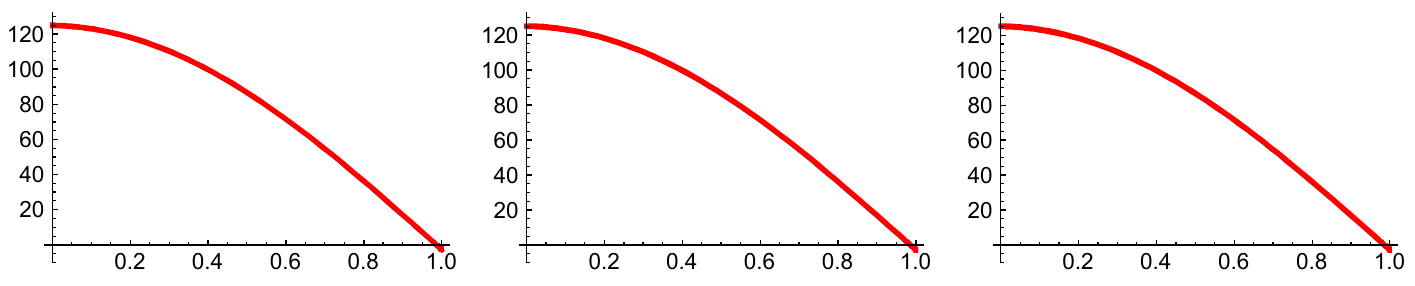}
\caption{Upper and lower approx.~of $k_4$ for $\kappa=0.5, 0.55, 0.56$ (left to right). Properties according to Lemma \ref{l:monotonicity_differences2}.}
\label{Plotk5}
\end{figure}

Next we prove the main observation leading to the convexity of $S(\xi,\cdot)-S(\eta,\cdot)$ on $[.1,.9]$.

\begin{lem} \label{lemposSdpri}For $\kappa \in [\eh,\kappa_0]$, we have
$$
S''(\xi,\cdot)-S''(\eta,\cdot)>0
$$
for $x\in [0.1,0.9]$.	
\end{lem}

\begin{proof}
We argue separately for the ten cases of Table \ref{table:Cases}, see also Figure \ref{PlotS''}.
	
\textbf{Case 1}:
$$
S''(\xi,x)-S''(\eta,x)=\kappa\Big(g''(x)+O\Big(\big(\frac{\kappa}{4}\big)^3\Big)\Big)>0 \text{ for } x\geq .1
$$
follows immediately from Lemma \ref{l:positivity_gdoubleprime}.

\textbf{Case 2}:
$$
S''(\xi,x)-S''(\eta,x)
=\kappa \Big(g''(x)+\big(\frac{\kappa}{4}\big)^2g''\big(\frac{x+1}{4}\big)+O\Big(\big(\frac{\kappa}{4}\big)^3\Big)\Big)>0 \text{ for } x\geq x_0
$$
where $x_0$ is given by Lemma \ref{l:positivity_gdoubleprime}. This follows from this Lemma and the fact that $g''(\frac{x+1}{4})>0$ for all $x\in[0,1]$.

\textbf{Case 3}:
$$
S''(\xi,x)-S''(\eta,x)
=\kappa\Big(g''(x)+\big(\frac{\kappa}{4}\big)^2\Big(g''\big(\frac{x+1}{4}\big)-g''\big(\frac{x}{4}\big)\Big)+O\Big(\big(\frac{\kappa}{4}\big)^3\Big)\Big)>0 \text{ for } x\geq x_0 .
$$
This follows from Lemma \ref{l:positivity_gdoubleprime} and the positivity of $k_3$ stated in Lemma \ref{l:monotonicity_differences}.

\textbf{Case 4}:
$$
S''(\xi,x)-S''(\eta,x)
=\kappa\Big(g''(x)+\frac{\kappa}{4}g''\big(\frac{x+1}{2}\big)+O\Big(\big(\frac{\kappa}{4}\big)^3\Big)\Big)>0 \text{ for } x\geq x_0 .
$$
This follows from the fact that $g''(\frac{x+1}{2})>0$ for $x\in[0,1]$ which is a consequence of Lemma \ref{l:positivity_gdoubleprime}.

\textbf{Case 5}:
$$
S''(\xi,x)-S''(\eta,x)
=\kappa\Big(g''(x)+\frac{\kappa}{4}g''\big(\frac{x+1}{2}\big)-\big(\frac{\kappa}{4}\big)^2g''\big(\frac{x}{4}\big)
+O\Big(\big(\frac{\kappa}{4}\big)^3\Big)\Big)>0 \text{ for } x\in[x_0,.9] .
$$
This follows from Lemma \ref{l:positivity_gdoubleprime} and Lemma \ref{l:monotonicity_differences2}.

\textbf{Case 6}:
$$
S''(\xi,x)-S''(\eta,x)=\kappa\Big(g''(x)+\frac{\kappa}{4}g''(\frac{x+1}{2})+(\frac{\kappa}{4})^2g''(\frac{x+1}{4}+\frac{1}{2})+O\Big(\big(\frac{\kappa}{4}\big)^3\Big)\Big)>0 \text{ for } x\geq x_0 .
$$
This follows from the positivity of $\frac{\kappa}{4}g''(\frac{x+1}{2})+(\frac{\kappa}{4})^2g''(\frac{x+1}{4}+\frac{1}{2})$ for all $x\in[0,1]$ which is a consequence of Lemma \ref{l:positivity_gdoubleprime}.

\textbf{Case 7}:
$$
S''(\xi,x)-S''(\eta,x)
=\kappa\Big(g''(x)
           +\frac{\kappa}{4}g''\big(\frac{x+1}{2}\big)
					 +\big(\frac{\kappa}{4}\big)^2\Big(g''\big(\frac{x+1}{4}+\frac{1}{2}\big)-g''\big(\frac{x}{4}+\frac{1}{2}\big)\Big)+O\Big(\big(\frac{\kappa}{4}\big)^3\Big)\Big)>0 \text{ for } x\geq x_0 .
$$
First note that $\frac{\kappa}{4}g''(\frac{x+1}{2})+(\frac{\kappa}{4})^2(g''(\frac{x+1}{4}+\frac{1}{2})-g''(\frac{x}{4}+\frac{1}{2}))>0$ for $x\in[0,1]$ due to Lemmas \ref{l:positivity_gdoubleprime} and \ref{l:monotonicity_differences}. An appeal to Lemma \ref{l:positivity_gdoubleprime} again proves positivity on $[x_0,1]$.

\textbf{Case 8}:
$$
S''(\xi,x)-S''(\eta,x)=\kappa\Big(g''(x)+\frac{\kappa}{4}\Big(g''\big(\frac{x+1}{2}\big)-g''\big(\frac{x}{2}\big)\Big)+O\Big(\big(\frac{\kappa}{4}\big)^3\Big)\Big)>0
$$
for $x\in [0.1,0.9]$. By Lemma \ref{l:monotonicity_differences}, we know that $k_1$ is decreasing, with $k_1(0)>0,$ and $k_1(1)<0$. It remains to calculate $k_1(.9)$ and remark that $g''(.9)+\frac{\kappa}{4} g''(.9)+(\frac{\kappa}{4})^2 k_1(.9)>0$.

\textbf{Case 9}:

$$
S''(\xi,x)-S''(\eta,x)
=
\kappa\Big(g''(x)+\frac{\kappa}{4}\Big(g''\big(\frac{x+1}{2}\big)-g''\big(\frac{x}{2}\big)\Big)+\big(\frac{\kappa}{4}\big)^2g''\big(\frac{x+1}{4}+\frac{1}{2}\big)+O\Big(\big(\frac{\kappa}{4}\big)^3\Big)\Big)>0
$$
for $x\in [0.1,0.9]$.
This follows from \textbf{Case 8} and the fact that $g''(\frac{x+1}{4}+\frac{1}{2})>0$ for $x\in [0.1,1]$.

\textbf{Case 10}:
\begin{align*}
S''(\xi,x)-S''(\eta,x)
&
=\kappa\Big(g''(x)
+\frac{\kappa}{4}\Big(g''\big(\frac{x+1}{2}\big)- g''\big(\frac{x}{2}\big)\Big)
\\
&\qquad \quad
+\big(\frac{\kappa}{4}\big)^2\Big(g''\big(\frac{x+1}{4}+\frac{1}{2}\big)-g''\big(\frac{x}{4}+\frac{1}{2}\big)\Big)+O\Big(\big(\frac{\kappa}{4}\big)^3\Big)\Big)>0
\end{align*}
for $x\in [x_0,0.9]$.
First observe that from Lemma \ref{l:monotonicity_differences}, we can deduce that $\frac{\kappa}{4}k_1(x)+(\frac{\kappa}{4})^2k_2(x)$ is decreasing in $x$, starting at $x=0$ with a positive value. Now an appeal to Lemma \ref{l:positivity_gdoubleprime} and a numerical verification of
$g''(.9)+\frac{\kappa}{4}k_1(.9)+(\frac{\kappa}{4})^2k_2(.9)>0$ show positivity on $[x_0,.9]$.
\end{proof}

\begin{figure}[ht]
\centering
\includegraphics[width=\textwidth]{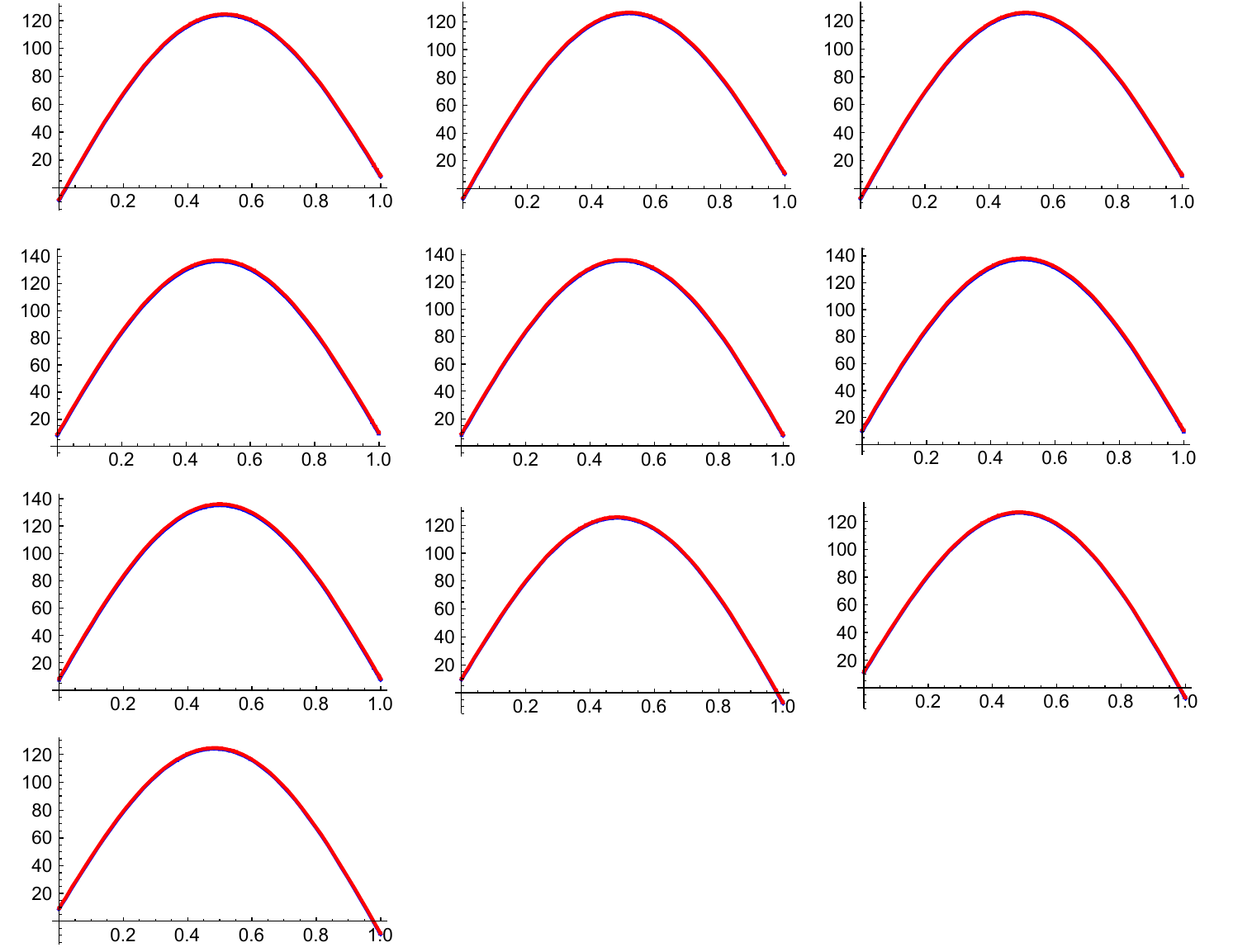}
\caption{Upper and lower approx.~of $S''(\xi,\cdot)-S''(\eta,\cdot)$ for $x\in[0,1]$ and $\kappa=0.55$ in ten cases of Table \ref{table:Cases} (Case 1 on top left, continuing in the usual fashion left to right, top to bottom). The maps are positive for $x\in[.1, .9].$}
\label{PlotS''}
\end{figure}

We shall now, again following the idea that the geometric behaviour of $S(\xi,\cdot)-S(\eta,\cdot)$ is just an individual perturbation of the one of $g$, show that $S'(\xi,\cdot)-S'(\eta,\cdot)$ is negative on $[0,.1]$ and positive on $[.9,1]$. This will allow us to conclude that $S(\xi,\cdot)-S(\eta,\cdot)$ possesses a unique local minimum. Our arguments will be similar to the ones above confirming the positivity of $S''(\xi,\cdot)-S''(\eta,\cdot)$, but simpler, starting with the representation of $S'(\xi,\cdot)-S'(\eta,\cdot)$ by equation \eqref{e:eqrepSprime}. As before, we shall start with an analysis of the relevant increments of $g'$ at different arguments.

\begin{lem}\label{l:sign_gprime}
We have for $x\in[0,1]$
\begin{align*}
l_1(x):=g'(\frac{x+1}{2})-g'(\frac{x}{2})
=& -8\pi^2\sum^{\infty}_{m=0}\Big(\frac{\kappa}{2}\Big)^{m}
\sin \Big( \frac{\pi}{2^{m+1}}\Big)\sin \Big( \frac{\pi}{2^{m+2}}\Big)\cos \Big( \frac{\pi}{2^{m+2}}(3+2x)\Big),
\\
l_2(x):=g'(\frac{x+1}{4}+\frac{1}{2})-g'(\frac{x}{4}+\frac{1}{2})
=& -8\pi^2\sum^{\infty}_{m=0}\Big(\frac{\kappa}{2}\Big)^{m}
\sin \Big( \frac{\pi}{2^{m+1}}\Big)\sin \Big( \frac{\pi}{2^{m+3}}\Big)\cos \Big( \frac{\pi}{2^{m+2}}(\frac{9}{2}+x)\Big),
\\
l_3(x):=g'(\frac{x+1}{4})-g'(\frac{x}{4})
=&-8\pi^2\sum^{\infty}_{m=0}\Big(\frac{\kappa}{2}\Big)^{m}
\sin \Big( \frac{\pi}{2^{m+1}}\Big)\sin \Big( \frac{\pi}{2^{m+3}}\Big)
\cos \Big( \frac{\pi}{2^{m+2}}(3+x)\Big).
\end{align*}
Moreover, $l_1,\cdots,l_3$ are strictly positive on $[0,1]$.
\end{lem}

\pf
As before, the claimed equations follow readily from trigonometric identities. We approximate the functions by their first term, and note that the remainder is small enough to not perturb its positivity on $[0,1].$
To obtain the latter, observe that $\cos$ is negative on $[\frac{\pi}{2}, \frac{3 \pi}{2}],$ and the intervals $[\frac{3\pi}{4}, \frac{5\pi}{4}], [\frac{9\pi}{8},\frac{11\pi}{8}], [\frac{3\pi}{4},\pi]$ are contained in $[\frac{\pi}{2}, \frac{3 \pi}{2}]$. See Figure \ref{Plotl1234}.  \epf

\begin{figure}[ht]
\centering
\includegraphics[width=\textwidth]{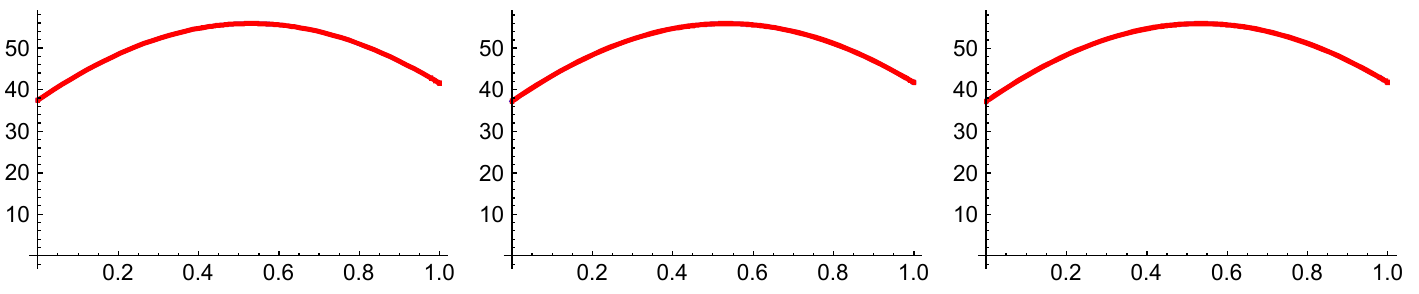}
\includegraphics[width=\textwidth]{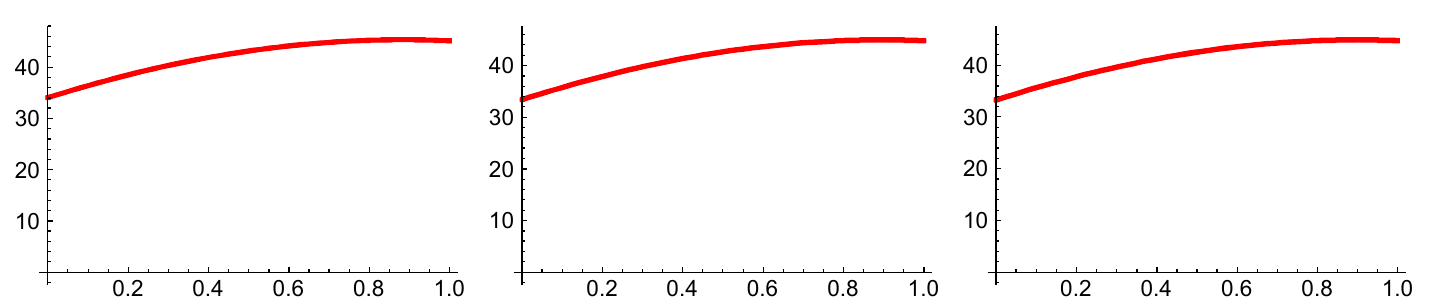}
\includegraphics[width=\textwidth]{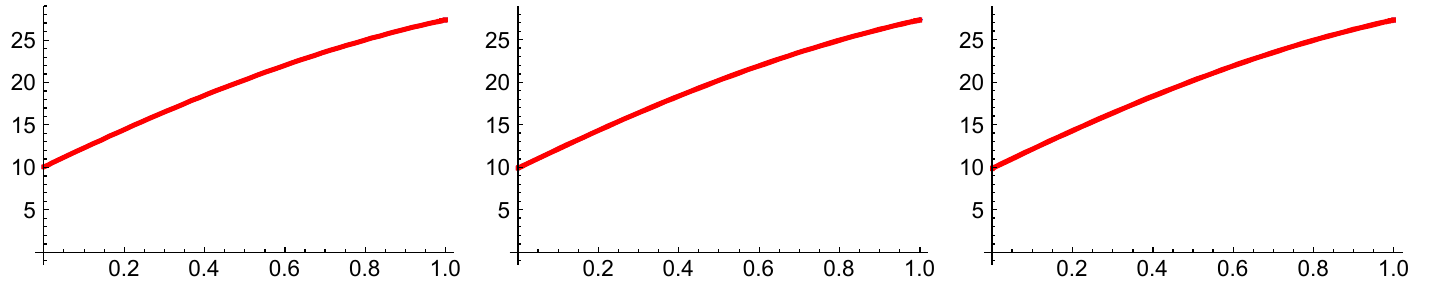}
\caption{Upper and lower approx.~of graph of $l_1,l_2,l_3$ (from top to bottom) for $\kappa=0.5, 0.55, 0.56$ (left to right). The functions are strictly positive on $[0,1]$ as per Lemma \ref{l:sign_gprime}.}
\label{Plotl1234}
\end{figure}

A somewhat different case is treated in the following lemma.

\begin{lem}\label{l:sign_gprime2}
The function
$$
l_4(x):= g'\Big(\frac{x+1}{2}\Big)-\frac{\kappa}{2}g'\Big(\frac{x}{4}\Big), \quad x\in[0,1],
$$
is strictly positive on $[0,1]$.
\end{lem}

\pf
We estimate by taking the first two terms in the series expansion of $g'(\frac{x+1}{2})$ and the first term of $g'(\frac{\kappa}{4})$. We obtain
\begin{align*}
l_4(x)
=&-4\pi^2\Big[\sin \Big(\frac{\pi}{2}(2+x)\Big) + \frac{\kappa}{4}(2-\sqrt{2}) \sin \Big(\frac{\pi}{2}(1+\frac{x}{2})\Big)\Big]
+O\Big(\big(\frac{\kappa}{2}\big)^2\Big)
\\
=&4\pi^2\Big[\sin \Big(\frac{\pi x}{2} \Big) + \frac{\kappa}{4}(2-\sqrt{2}) \cos \Big( \frac{\pi x}{4}\Big)\Big]
+O\Big(\big(\frac{\kappa}{2}\big)^2\Big)
\end{align*}
This function is clearly positive on $[0,1]$ (see Figure \ref{Plotl5}), since for the given range of $\kappa$ the error term is small enough. \epf

\begin{figure}[h!tbp]
\centering
\includegraphics[width=\textwidth]{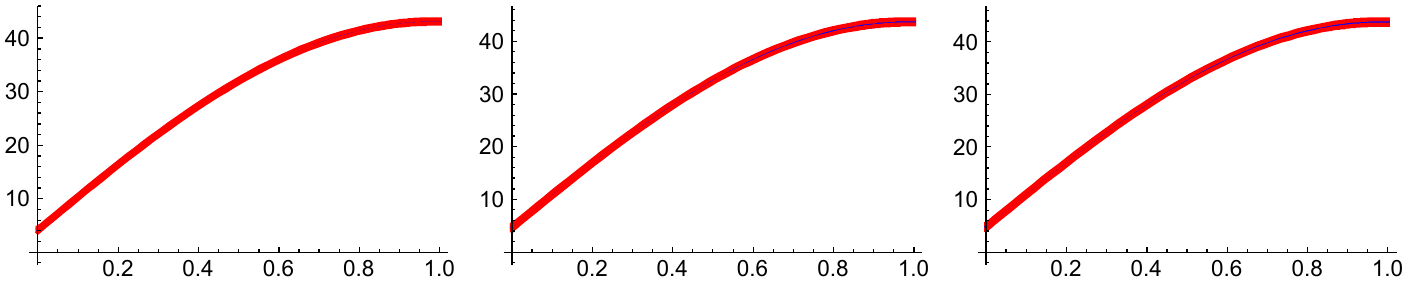}
\caption{Upper and lower approx.~of graph of $l_4$ for $\kappa=0.5, 0.55, 0.56$ (left to right). The function is strictly positive as per Lemma \ref{l:sign_gprime2}.}
\label{Plotl5}
\end{figure}

Next we state the result on the sign of $S'(\xi,\cdot)-S'(\eta,\cdot)$ on $[0,.1]$ and $[.9,1]$.

\begin{lem} \label{lemposSdpri2}
For $\kappa \in [\eh,\kappa_0]$, we have
$$
S'(\xi,\cdot)-S'(\eta,\cdot)<0\quad\mbox{on}\quad[0,.1],\quad\quad S'(\xi,\cdot)-S'(\eta,\cdot)>0\quad\mbox{on}\quad[.9,1].
$$	
\end{lem}	

\pf The inspection of the ten cases of Table \ref{table:Cases} is analogous to the one in the proof of Lemma \ref{lemposSdpri}, but simpler. The essential observation is that $g'|_{[0,.1]}$ is negative and very small compared to the positive contributions coming from the additional terms discussed in Lemmas \ref{l:sign_gprime} and \ref{l:sign_gprime2}. Hence negativity is preserved on $[0,.1]$ in all cases. The argument for positivity on $[.9,1]$ is even simpler. We omit further details.\epf

\begin{figure}[h!tbp]
\centering
\includegraphics[width=\textwidth]{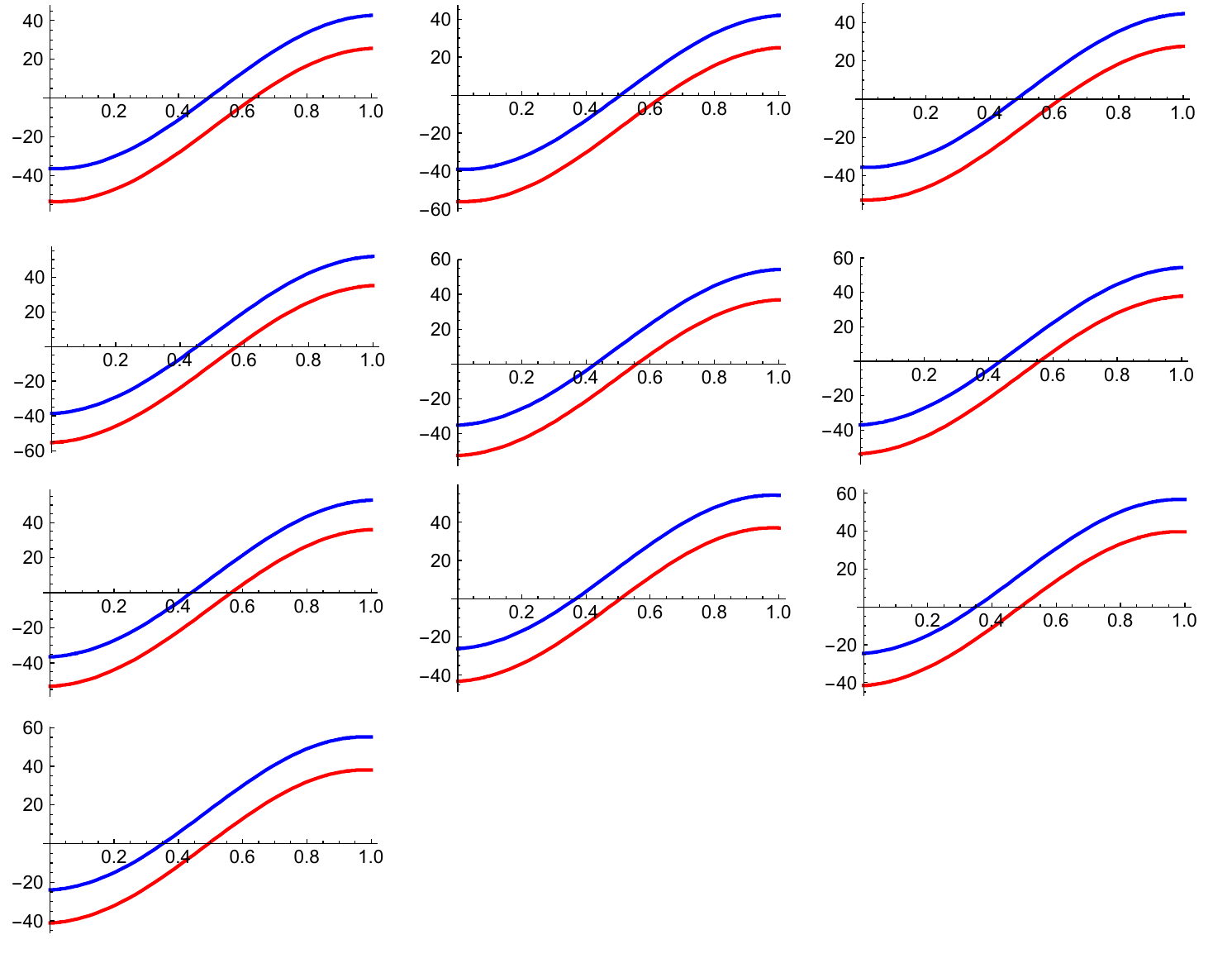}
\caption{Upper and lower approx.~of $S'(\xi,\cdot)-S'(\eta,\cdot)$ for $x\in[0,1]$ for $\kappa=0.55$ in ten cases of Table \ref{table:Cases} (Case 1 on top left, continuing in the usual fashion left to right, top to bottom). The maps are negative for $x\in[0,.1]$ and positive for $x\in[.9,1]$ in accordance with Lemma \ref{lemposSdpri2}.}
\label{PlotS'}
\end{figure}

				\section{The relationship between $\rho, \hat{\rho}$ and Lebesgue measure}\label{s:rho_rhohat}
				
				In this section we shall exploit the scaling properties of $S$, more precisely its self affinity in order to express $\rho$ in terms of $\hat{\rho}$. Recall that $\rho$ is the image measure of the three-dimensional Lebesgue measure by the map
				$$[0,1]^3\ni (x,\xi,\eta)\mapsto S(\xi,x)-S(\eta,x),$$
namely, for any Borel set $A\subset \R$, we define
				\[
				\rho(A)=\lambda^3\Big(\big\{ (x,\xi,\eta)\in [0,1]^3: S(\xi,x)-S(\eta,x) \in A\big\}\Big).
				\]
				and $\hat{\rho}(\cdot) = \rho(\cdot\cap\{\eh<|\xi-\eta|\}).$
				By its very definition, $\hat{\rho}$ lives on the set of pairs $(\xi, \eta)\in[0,1]^2$ for which $\eh<|\xi-\eta|$. On this set, transversality of $S$ will allow a comparison of $\hat{\rho}$ with the Lebesgue measure. This will finally lead to conclusions about the regularity of the SBR measure. In the following formula, $\rho$ is shown to be a weighted average of expansions measured by $\hat{\rho}.$
				
				\begin{pr}\label{p:rho-rhohat}
					For Borel sets $A$ on the real line, we have
					$$\rho(A) = \sum_{n=0}^\infty 2^{-n} \hat{\rho}(\kappa^{-n} A).$$
				\end{pr}

				\pf We have using Lemma \ref{l:action_B}
				\bea
				\rho(A) &=& \sum_{n=0}^\infty \lambda^3\Big(\Big\{(\xi,\eta,x): S(\xi,x)-S(\eta,x)\in A, 2^{-(n+1)}<|\xi-\eta|\le 2^{-n}\Big\}\Big)\\
				&=& \sum_{n=0}^\infty \lambda^3\Big(\Big\{(\xi,\eta,x): S(\xi,x)-S(\eta,x)\in A, \eh<|B_1^{n}(\xi,x)-B_1^{n}(\eta,x)|\le 1\Big\}\Big)\\
				&=& \sum_{n=0}^\infty \lim_{\epsilon\to 0} \frac{1}{\lambda^4(D_\epsilon)}
				                     \lambda^4\Big(\Big\{(\xi,\eta,x,y): S(\xi,x)-S(\eta,y)\in A,\\
&&\hspace{5cm} \eh<|B_1^{n}(\xi,x)-B_1^{n}(\eta,y)|\le 1, |x-y|\le \epsilon\Big\}\Big),
				\eea
where for $\epsilon>0$ we let $D_\epsilon = \{(\xi,\eta,x,y):|x-y|\le \epsilon\}.$
				We next use the invariance of $B$ (see \eqref{eq:InvarianceBakerB}) to estimate term $n$ of the preceding series.
				Then for $n\ge 0$
				\bea
				&&
				\lim_{\epsilon\to 0} \frac{1}{\lambda^4(D_\epsilon)}
				                     \lambda^4\Big(\Big\{(\xi,\eta,x,y): S(\xi,x)-S(\eta,y)\in A,\\
&&\hspace{5cm} \eh<|B_1^{n}(\xi,x)-B_1^{n}(\eta,y)|\le 1, |x-y|\le \epsilon\Big\}\Big),
\\
				&&\hspace{.5cm} = \lim_{\epsilon\to 0} \frac{1}{\lambda^4(D_{\epsilon})} \lambda^4\Big(\Big\{(\xi,\eta,x,y): S(B^{-n}(\xi,x))-S(B^{-n}(\eta,y))\in A,
				\\
				&&\hspace{5cm}\eh<|\xi-\eta|\le 1, |B^{-n}_2(\xi,x)-B^{-n}_2(\eta,y)|\le \epsilon\Big\}\Big)
				\\
				&&\hspace{.5cm}= \lim_{\epsilon\to 0} \frac{\lambda^4(D_{2^{-n}\epsilon})}{\lambda^4(D_\epsilon)} \frac{1}{\lambda^4(D_{2^{-n}\epsilon})}
				\lambda^4\Big(\Big\{ (\xi,\eta,x,y): S(B^{-n}(\xi,x))-S(B^{-n}(\eta,y))\in A,
				\\
        &&\hspace{5cm} \eh<|\xi-\eta|\le 1, |x-y|\le 2^{-n}\epsilon\Big\}\Big)
				\\
				&&\hspace{.5cm}= 2^{-n} \lambda^3\Big(\Big\{ (\xi,\eta,x): S(B^{-n}(\xi,x))-S(B^{-n}(\eta,x))\in A, \eh<|\xi-\eta|\le 1\Big\}\Big).
				\eea

				Now we apply Lemma \ref{l:scaling_G} to transform term $n$ in the preceding chain of equations into
$$
\lambda^3\Big(\Big\{ (\xi,\eta,x): \kappa^{n}(S(\xi,x)-S(\eta,x))\in A, \eh<|\xi-\eta|\le 1\Big\}\Big) = \hat{\rho}\big(\kappa^{-n} A\big).
$$
This implies the claimed equation. \epf


To abbreviate, define
				$$f_{\xi,\eta} (\cdot) = S(\xi,\cdot)-S(\eta,\cdot).$$
				We know from section \ref{s:transversality} that for $\kappa\le\kappa_0$, with a $\kappa_0 \in [0.55, 0.56]$, the map $f_{\xi,\eta}$, restricted to $[0,1]$, possesses a unique negative (global) minimum. For $y\in\R$, denote by $f^{-1}_{\xi,\eta}(y)$ the set of at most two points $x\in[0,1]$ satisfying $f_{\xi,\eta}(x)=y.$ We can state the following proposition.
				
				\begin{pr}\label{p:density_rhohat}
					Let $\kappa\le \kappa_0.$ Then $\hat{\rho}$ is absolutely continuous with respect to the Lebesgue measure with density
					$$\phi(y)= \int_{\{\eh<|\xi-\eta|\}} \sum_{x\in f^{-1}_{\xi,\eta}(y)}\frac{1}{|f'_{\xi,\eta}(x)|} d\xi d\eta, \quad y\in\R.$$
				\end{pr}
				
				\pf Let $y\in\R, \xi,\eta\in[0,1].$ Then the $(\xi,\eta)$-section of $\hat{\rho}([y-\epsilon, y+\epsilon])$ is given by
				$\int_0^1 1_{[y-\epsilon, y+\epsilon]}(S(\xi,x)-S(\eta,x)) dx,$ and therefore by the regularity of $f_{\xi,\eta}$
				$$
				\lim_{\epsilon\to 0} \frac{1}{2\epsilon} \hat{\rho}\big([y-\epsilon, y+\epsilon]\big) = \sum_{x\in f^{-1}_{\xi,\eta}(y)} \frac{1}{|f'_{\xi,\eta}(x)|}.
				$$
				Therefore the desired formula for the density of $\hat{\rho}$ at $y$ follows from integrating the expression obtained in $(\xi,\eta)\in[0,1]^2$ and Fubini's theorem.
\epf
				
				Combining the preceding two propositions we obtain a similar absolute continuity statement for $\rho.$
				
				\begin{co}\label{c:density_rho}
					Let $\kappa\le \kappa_0.$ Then $\rho$ is absolutely continuous with respect to the Lebesgue measure with density
					$$\R\ni y\mapsto \sum_{n=0}^\infty \gamma^n \sum_{x\in f_{\xi,\eta}^{-1}(\kappa^{-n}y)}\frac{1}{|f'_{\xi,\eta}(x)|} d\xi d\eta.$$
				\end{co}
				
				\pf Let $y\in\R.$ Then the claimed formula follows by applying Proposition \ref{p:rho-rhohat}, calculating
				$$\lim_{\epsilon\to 0} \sum_{n=0}^\infty 2^{-n} \hat{\rho}\big([\kappa^{-n}(y-\epsilon), \kappa^{-n}(y+\epsilon)]\big),$$
				in combination with Proposition \ref{p:density_rhohat} and remarking that $\frac{1}{2 \kappa} = \gamma.$ \epf
				
				\section{The absolute continuity of the SBR measure}\label{s:abs_continuity}
				
In this section we will finally draw our conclusions from the preceding two sections. In fact, we will derive a sufficient criterion for the absolute continuity of the SBR measure from Corollary \ref{c:density_rho}. For this purpose we consider the Fourier transforms of the marginals $\mu_x$, $x\in[0,1]$,  of the SBR measure $\mu$ defined in \eqref{eq:SBRforGamma}. Let
$$\phi_x(u) = \int_{\R} \exp(i u y) \mu_x(\dd y),\quad u\in\R.$$
By definition of $\mu$ and the integral transform theorem we have
$$\phi_x(u) = \int_0^1 \exp\big(i u S(\xi,x)\big) \dd\xi,\quad u\in\R, x\in[0,1].$$
To prove the absolute continuity of $\mu_x$ we have to prove that $\phi_x$ is square integrable on $\R$. Therefore, to prove that $\mu$ is absolutely continuous, it will be sufficient to prove
\bea
\int_0^1 \int_\R |\phi_x(u)|^2 \dd u \dx
&=& \int_\R \int_{[0,1]^3} \exp\Big(i u \big(S(\xi, x) - S(\eta,x)\big)\Big) \dx \dd\xi \dd\eta  \dd u\\
&=& \int_\R \int_\R \exp(i u x) \rho(dx)du < \infty.
\eea

\begin{thm}\label{t:SBR_ac}
Let $\kappa\le\kappa_0.$ Assume that
$$(\xi,\eta,y)\mapsto \sum_{x\in f_{\xi,\eta}^{-1}(y)} \frac{1}{|f'_{\xi,\eta}(x)|}$$
is bounded and continuous (in $y$) at $0$.
Then for almost every $x\in[0,1]$ the function
$$\xi\mapsto S(\xi,x)$$
has an absolutely continuous law with respect to the Lebesgue measure with a square integrable density. In particular, the SBR measure \eqref{eq:SBRforGamma} is absolutely continuous with respect to the Lebesgue measure and possesses a square integrable density.
\end{thm}

\pf
By Proposition \ref{p:density_rhohat}, the integral transformation formula and noting that $\frac{1}{2 \kappa} = \gamma$, we may write
\bea
\int_\R \int_\R \exp(i u x) \rho(dx) du
&=& \sum_{n=0}^\infty 2^{-n} \int_\R \int_\R \exp(i  u y) \hat{\rho}(\kappa^{-n} dy) \du
\\
&=& \int_\R \int_\R \sum_{n=0}^\infty 2^{-n} \exp(i u \kappa^n y) \hat{\rho}(dy) du\\
&=&\int_\R \int_\R \sum_{n=0}^\infty \gamma^n \exp(i u y) \hat{\rho}(dy) du\\
&=& \frac{1}{1-\gamma} \int_\R \int_\R \exp(i u y) \hat{\rho}(dy) du\\
&=& \frac{1}{1-\gamma} \int_\R\int_\R \exp(i u y) \int_{[0,1]^2} \sum_{x\in f_{\xi,\eta}^{-1}(y)} \frac{1}{|f'_{\xi,\eta}(x)|} d\xi d\eta dy du.
\eea
Abbreviate
$$g(y) = \frac{1}{1-\gamma} \sum_{x\in f_{\xi,\eta}^{-1}(y)} \frac{1}{|f'_{\xi,\eta}(x)|} d\xi d\eta.$$
Then by hypothesis and the dominated convergence $y\mapsto g(y)$ is bounded and continuous at $0$.
We have to show that
$$\limsup_{K\to\infty}  \int_{-K}^K \int_\R \exp(i u y)du\,g(y) dy <\infty.$$
Recall that $\hat{\rho}$ is antisymmetric with respect to reflection at the origin and has compact support $[-L,L].$ Hence we have
\bea
\int_{-K}^{K} \int_\R \exp(i u y) \du g(y) dy
&=& \int_{-K}^{K}\int_{-L}^L \exp(i u y) g(y) dy \dd u\\
&=& 2\int_{-L}^L \int_0^{K} \cos(u y) \dd u g(y) dy\\
&=& 2 \int_{-L}^L \frac{\sin(K y)}{y}  g(y) dy.
\eea
Now note that for a bounded function $g$ on $[-L,L]$,
$$\limsup_{K\to\infty} \int_{-L}^L g(y) \frac{\sin(K y)}{y}<\infty,$$
provided $g$ is continuous at $0$. This implies the claimed absolute continuity.
\epf
				
				Transversality guarantees that the sufficient condition of the Theorem is satisfied.
				\begin{co}\label{c:SBR_ac_transversality}
					Let $\kappa\le\kappa_0$. Then the SBR measure is absolutely continuous with respect to Lebesgue measure. Its density is square integrable. 
				\end{co}
				
				\pf If $S$ satisfies the transversality property, then we know that there exists $\epsilon>0$ such that $\inf_{(x,\xi,\eta)\in[0,1]^3}|V_{\xi,\eta}(x)|>\epsilon,$ for
				$$V_{\xi,\eta}(x) = \big(\, f_{\xi,\eta}(x), f'_{\xi,\eta}(x)\, \big).$$ In this case
$$y\mapsto \int_{[0,1]^2} \sum_{x\in f_{\xi,\eta}^{-1}(y)} \frac{1}{|f'_{\xi,\eta}(x)|} d\xi d\eta$$
is continuous at $0$ and
				$$\sum_{x\in f_{\xi,\eta}^{-1}(0)} \frac{1}{|f'_{\xi,\eta}(x)|} \le \frac{2}{\epsilon}.$$
				This implies that the conditions of Theorem \ref{t:SBR_ac} are satisfied. To prove the claim on the density, observe that $g$ is the convolution $f*f$ of the density $f$, and hence $\int_\R f^2(y) dy = f*f(0) = g(0),$ which is finite by hypothesis. \epf

{\bf Remark:}\\ One may conjecture that the explicit formula for the density of the SBR measure given in Proposition \ref{p:density_rhohat} should provide the square integrability just proved with Theorem \ref{t:SBR_ac}. In fact, the parabolic structure of the denominators of the terms in the formula for the density gives rise to the statement that singularities of the density function are of the form $y\mapsto {1}/{\sqrt{|y-y_0|}}$ for certain $y_0\in\R.$ This function is $p$-integrable for any $p< 2,$ but just not square integrable. We expect that the subtle dependence of $y_0$ on $\xi$ may contain the key to a direct derivation of square integrability from Proposition \ref{p:density_rhohat}.

%



\begin{thebibliography}{99}

					\bibitem{bahouri11} H.~Bahouri, J.-Y.~Chemin, R.~Danchin. \emph{Fourier Analysis and Nonlinear Partial Differential Equations.} Springer: Berlin 2011.

\bibitem{baranski15survey} K. Baranski. \emph{Dimension of the graphs of the Weierstrass-type functions.} Fractal Geometry and Stochastics V. Springer International Publishing (2015), 77-91.
					
					\bibitem{baranski02} K.~Baranski. \emph{On the complexification of the Weierstrass non-differentiable function.} Anales-Acadamiae Scientarium Fennicae Mathematica. Vol. 27 (2002), no. 2. Academia Scientarium Fennica.
					
					\bibitem{baranski12} K.~Baranski. \emph{On the dimension of graphs of Weierstrass-type functions with rapidly growing frequencies.} Nonlinearity 25 (2012), no. 1, 193--209.

					\bibitem{baranski14published} K.~Baranski, B.~Barany, J.~Romanova. \emph{On the dimension of the graph of the classical Weierstrass function.} Advances in Mathematics 265 (2014), 32--59.
					

					
\bibitem{catelliergubinelli16} R.~Catellier, M.~Gubinelli. \emph{Averaging along irregular curves and regularisation of ODEs.} Stochastic Process. Appl. 126 (2016), no. 8, 2323--2366.

\bibitem{carvalho2011} A.~Carvalho. \emph{Hausdorff dimension of scale-sparse Weierstrass-type functions.}
Fund. Math. 213 (2011), no. 1, 1--13.

\bibitem{gubinelli2004} M.~Gubinelli
\emph{Controlling rough paths,} J. Funct. Anal. 216 (2004), no. 1, 86--140.

\bibitem{gubinelliimkellerperkowski2015}
M. Gubinelli, P. Imkeller, N. Perkowski. \emph{Paracontrolled distributions and singular PDEs.} Forum Math.Pi - Vol. 3 (2015), e6, 75.

\bibitem{gubinelliimkellerperkowski2016} M.~Gubinelli, P.~Imkeller, N.~Perkowski. \emph{A Fourier approach to pathwise stochastic integration.} Electron. J. Probab. 21 (2016), no. 2, 1--37.

\bibitem{hardy1916} G.~H.~Hardy. \emph{Weierstrass's non-differentiable function.} Trans. Amer. Math. Soc 17.3 (1916), 301--325.

\bibitem{hunt1998} B.~Hunt. \emph{The Hausdorff dimension of graphs of Weierstrass functions.} Proc. Amer. Math. Soc. 126.3 (1998), 791–-800.


\bibitem{IdR2018} P.~Imkeller, O.~Menoukeu-Pamen, G.~dos Reis. \emph{On the Hausdorff dimension of a 2-dimensional Weierstrass curve} (2018). arXiv:1806.09585.
					
\bibitem{imkellerproemel15} P.~Imkeller, D.~Pr\"omel. \emph{Existence of L\'evy's area and pathwise integration}, Communications on Stochastic Analysis 9 (2015), no. 1, 93--111.

\bibitem{keller17-Publication-of-2015}
G.~Keller. \emph{A simpler proof for the dimension of the graph of the classical Weierstrass function.} Annales de l'Institut Henri Poincar\'e, Probabilit\'es et Statistiques 53 (2017), no. 1, 169--181.

\bibitem{ledrappier92} F.~Ledrappier. \emph{On the dimension of some graphs.} Contemp. Math. 135 (1992), 285--293.

\bibitem{mortersperes10} P.~M\"orters, P.~Peres. \emph{Brownian motion.} Vol. 30. Cambridge University Press: Cambridge 2010.


\bibitem{shen2018} W.~Shen. \emph{Hausdorff dimension of the graphs of the classical Weierstrass functions.} Math. Z. 289 (2018), no. 1--2, 223--266.

\bibitem{tsujii01} M.~Tsujii. \emph{Fat solenoidal attractors.} Nonlinearity 14 (2001), no. 5, 1011--1027.

				\end{thebibliography}
			\end{document}